\documentclass[12pt]{article}
\usepackage[utf8]{inputenc}
\usepackage[margin=2.5cm]{geometry}
\usepackage{amsthm}
\usepackage{adjustbox}
\usepackage{enumitem}
\setlist[itemize]{topsep=0pt}
\setlist[enumerate]{topsep=0pt}
\usepackage{amsmath}
\usepackage{amsfonts}
\usepackage{ amssymb }
\usepackage[title]{appendix}
\usepackage{adjustbox}
\usepackage{multirow}
\usepackage{multicol}
\usepackage{changepage}
\usepackage{color}
\usepackage{array}
\usepackage{graphicx}
\usepackage{hyperref}
\usepackage{mathtools}
\usepackage{xcolor}
\usepackage{tikz}
\usetikzlibrary{decorations.markings}
\usepackage{bm}
\usepackage{amsbsy}
\usepackage{titlefoot}
\usepackage{titlesec}
\titleformat*{\section}{\centering\bfseries\small\MakeUppercase}
\titleformat*{\subsection}{\normalsize\bfseries}
\titlespacing\section{0pt}{6pt plus 4pt minus 2pt}{0pt plus 2pt minus 2pt}
\titlespacing{\subsection}{0pt}{-2pt}{-4pt}
\titlespacing{\subsubsection}{0pt}{-2pt}{-4pt}
\usepackage{thm-restate}
\usepackage[capitalise]{cleveref}

\newtheoremstyle{exampstyle}
  {\topsep} 
  {0pt} 
  {\slshape} 
  {} 
  {\bfseries} 
  {.} 
  {.5em} 
  {} 
\theoremstyle{exampstyle}
\newtheorem{thm}{Theorem}[section]
\newtheorem{cor}[thm]{Corollary}
\newtheorem{lemma}[thm]{Lemma}

\newtheorem{prop}[thm]{Proposition}

\newtheoremstyle{def}
  {\topsep} 
  {0pt} 
  {} 
  {} 
  {\bfseries} 
  {.} 
  {.5em} 
  {} 
\theoremstyle{def}
\newtheorem{defn}[thm]{Definition}

\newtheorem{rmk}[thm]{Remark}
\newtheorem{exmp}[thm]{Example}
\newtheorem{question}[thm]{Question}

\newcommand{\comm}[1]{}
\newcommand{\Q}{\mathbb{Q}}
\newcommand{\Z}{\mathbb{Z}}

\newcommand{\RAAG}{A_{\Gamma}}

\newcommand{\myRed}[1]{\textbf{\textcolor{red}{#1}}}

\title{\large{\uppercase{\textbf{Conjugacy problem in virtual right-angled Artin groups}}}}
\author{Gemma Crowe}
\date{}

\begin{document}

\maketitle
\begin{abstract}
    In this paper we solve the conjugacy problem for several classes of virtual right-angled Artin groups, using algebraic and geometric techniques. We show that virtual RAAGs of the form $A_{\phi} = \RAAG \rtimes_{\phi} \Z/m\Z$ are $\mathrm{CAT}(0)$ when $\phi \in \mathrm{Aut}(\RAAG)$ is length-preserving, and so have solvable conjugacy problem. The geometry of these groups, namely the existence of contracting elements, allows us to show that the conjugacy growth series of these groups is transcendental. Examples of virtual RAAGs with decidable conjugacy problem for non-length preserving automorphisms are also studied. Finally, we solve the twisted conjugacy problem in RAAGs with respect to length-preserving automorphisms, and determine the complexity of this algorithm in certain cases. \\

    2020 Mathematics Subject Classification: 20E22, 20F10, 20F36, 20F67.
\end{abstract}
\unmarkedfntext{\emph{Keywords}: Conjugacy problem, twisted conjugacy problem, right-angled Artin groups, conjugacy growth}

\section{Introduction}
The conjugacy problem for groups has been studied extensively for the past century, and asks whether for a group $G$ with finite generating set $X$, there exists an algorithm to determine if any two words $u,v \in \left(X \cup X^{-1}\right)^{\ast}$ represent conjugate elements in $G$. One line of enquiry in recent years has been understanding the decidability of this problem in group extensions. This was first motivated by the study of the conjugacy problem in free-by-cyclic groups \cite{bogopolski_conjugacy_2006}, which was later extended to more general group extensions in \cite{bogopolski_orbit_2009}. A key tool from this paper in determining both decidable and undecidable results in group extensions is an algebraic property known as \emph{twisted conjugacy}. 

We say two elements $u,v \in G$ are \emph{twisted conjugate}, by some automorphism $\phi \in \mathrm{Aut}(G)$, if there exists an element $w \in G$ such that $v = \phi(w)^{-1}uw$. The \emph{twisted conjugacy problem} (TCP) asks whether there exists an algorithm to determine if two elements, given as words over $X \cup X^{-1}$, are twisted conjugate by some automorphism $\phi \in \mathrm{Aut}(G)$. Note that a positive solution to the twisted conjugacy problem implies a positive solution to the conjugacy problem, however the converse is not necessarily true \cite[Corollary 4.9]{bogopolski_orbit_2009}. In \cite{bogopolski_orbit_2009}, it was shown that the conjugacy problem in a group extension is directly linked to the twisted conjugacy problem in the base group. This motivated the study of the twisted conjugacy problem in several classes of groups \cite{ burillo_conjugacy_2016, cox_twisted_2017, crowe_twisted_2024, crowe_twisted_2024_2, gonzalez-meneses_twisted_2014}. In particular, the twisted conjugacy problem is decidable for both free and free abelian groups \cite{ciobanu_fixed_2022, ventura_multiple_2021}, which motivates the groups of interest for this paper, namely \emph{right-angled Artin groups} (RAAGs).

RAAGs, which we denote by $A_{\Gamma}$, are groups defined by finite simple graphs, with vertex sets $V(\Gamma)$ and edge sets $E(\Gamma)$. RAAGs are generated by the vertices of the graph, and relations are precisely of the form $ab = ba$, for every pair of generators $a, b \in V(\Gamma)$ such that $\{a,b\} \in E(\Gamma)$. A vast array of results have been shown for these groups \cite{charney_introduction_2007}, including a linear-time solution to the conjugacy problem \cite{crisp_conjugacy_2009}. The aim of this paper is to solve the conjugacy problem in virtual RAAGs, that is, groups which contain a finite index subgroup isomorphic to a RAAG. Note decidability of the conjugacy problem does not necessarily pass through finite group extensions \cite{collins_conjugacy_1977}. Primarily we will focus on finite cyclic extensions of RAAGs, which we define as groups with presentation
\begin{equation}\label{semidirect1}
    A_{\phi} = A_{\Gamma} \rtimes_{\phi} \Z/m\Z = \langle V(\Gamma), t \mid R(A_{\Gamma}), \; t^{m} = 1, \;  t^{-1}xt = \phi(x) \; (x \in V(\Gamma)) \; \rangle,
\end{equation}
where $\phi \in \mathrm{Aut}(\RAAG)$ is of finite order. 

In the first half of this paper, we study algebraic and geometric properties of virtual RAAGs. All RAAGs are $\mathrm{CAT}(0)$ groups \cite{charney_introduction_2007} (see \cref{sec:CAT0}), and we investigate when this geometric property passes to virtual RAAGs. Similar to conjugacy, the $\mathrm{CAT}(0)$ property is not a quasi-isometry invariant \cite[Caveat 7.6.7]{loh_geometric_2017}, however we find that this geometric property does hold for any virtual RAAG $A_{\phi}$, as in \cref{semidirect1}, when $\phi \in \mathrm{Aut}(\RAAG)$ is length-preserving. Here automorphisms of RAAGs can be defined as either length-preserving, where the length of any geodesic is preserved, and non-length preserving otherwise.

\begin{restatable*}{prop}{CATzero}\label{thm: CAT(0) extension RAAG}
    The group $A_{\phi} = \RAAG \rtimes_{\phi} \Z/m\Z$, as in \cref{semidirect1}, is a $\mathrm{CAT}(0)$ group, when $\phi \in \mathrm{Aut}(\RAAG)$ is length-preserving.  
\end{restatable*}
 \cite[III.$\Gamma$, Theorem 1.12]{bridson_metric_1999} then implies the following result.
\begin{restatable*}{thm}{CATzeroresult}\label{cor:CP solvable RAAGs len p CAT0}
    The conjugacy problem is solvable for all virtual RAAGs of the form $A_{\phi} = A_{\Gamma} \rtimes_{\phi} \Z/m\Z$, as in \cref{semidirect1}, when $\phi \in \mathrm{Aut}(\RAAG)$ is length-preserving.
\end{restatable*}
The geometry of these groups allows us to show that certain classes of virtual RAAGs contain a \emph{contracting element} (see \cref{defn:contracting}). This allows us to determine the \emph{conjugacy growth series} (see \cref{defn:conjugacy growth function series}) for some virtual RAAGs, using \cite[Corollary 1.8]{gekhtman_counting_2022}

\begin{thm}\label{thm:combine RAAG RACG growth}(see \cref{cor:extension transc contract} and \cref{thm:RACG conj growth})
Let $A_{\phi} = \RAAG \rtimes_{\phi} \Z/m\Z$ be a virtual RAAG as in \cref{semidirect1}, such that $\RAAG$ is not a direct product or cyclic, and $\phi \in \mathrm{Aut}(\RAAG)$ is length-preserving. Then the conjugacy growth series of $A_{\phi}$ is transcendental, with respect to the word metric. This result also holds for right-angled Coxeter groups which are not virtually a direct product.
\end{thm}

We also investigate finite extensions of RAAGs with respect to non-length preserving automorphisms. In several cases, virtual RAAGs are isomorphic to a graph product of cyclic groups, which have solvable conjugacy problem by \cite[Theorem 3.24]{green_graph_1990}.
\begin{restatable*}{cor}{mixCPRAAGs}\label{cor:CP solvable mix}
    There exists virtual RAAGs $A_{\phi} = \RAAG \rtimes_{\phi} \Z/m\Z$ as in \cref{semidirect1}, with respect to both length and non-length preserving automorphisms, which have solvable conjugacy problem. 
\end{restatable*}
In the second half of this paper, we provide an alternative proof to \cref{cor:CP solvable RAAGs len p CAT0}, by constructing an implementable algorithm to solve the twisted conjugacy problem in RAAGs, when our automorphism $\phi \in \mathrm{Aut}(\RAAG)$ is length-preserving. We also show that the complexity matches that of the conjugacy problem when our automorphism consists of inversions only.
\begin{restatable*}{thm}{TCPRAAG}\label{thm:len p RAAGs pilings}
    The twisted conjugacy problem $\mathrm{TCP}_{\phi}(\RAAG)$ is solvable for all RAAGs, when $\phi \in \mathrm{Aut}(\RAAG)$ is length-preserving. Moreover, when $\phi \in \mathrm{Aut}(\RAAG)$ is a composition of inversions, it is decidable in linear time, on a RAM machine, whether two elements are twisted conjugate in $\RAAG$. 
\end{restatable*}
\comm{
We note that the results from this paper can be extended to right-angled Coxeter groups, which have the same group presentation as RAAGs, with the added relation that each generator has order two.}
The structure of this paper is as follows. After providing necessary definitions and results in \cref{sec:prelims}, we provide a survey of algebraic and geometric results in \cref{sec:survey} related to the conjugacy problem in virtual RAAGs. Finally in \cref{sec:pilings}, we provide a summary of the algorithm given in \cite{crisp_conjugacy_2009} to solve the conjugacy problem in RAAGs, before adapting these ideas to solve the twisted conjugacy problem in RAAGs, with respect to length-preserving automorphisms.

\section{Preliminaries}\label{sec:prelims}
\subsection{Notation}
Throughout this paper, groups are finitely generated. For a subset $S$ of a group, we let $S^{\pm} = S \cup S^{-1}$, where $S^{-1} = \{s^{-1} \mid s \in S \}$. Let $X$ be a finite set, and let $X^{\ast}$ denote the set of all finite words over $X$. For any word $w \in X^{\ast}$, we let $l(w)$ denote the word length of $w$ over $X$. For a group element $g \in G$, we define the \emph{length} of $g$, denoted $| g |_{X} $, to be the length of a shortest representative word for the element $g$ over $X$, i.e.\ $|g| = \min \{ l(w) \mid w\in X^*, \; w =_G g\}$ (if $X$ is fixed or clear from the context, we write $|g|$). A word $w \in X^{\ast}$ is called a \emph{geodesic} if $l(w) = |\pi(w)|$, where $\pi \colon X^{\ast} \rightarrow G$ is the natural projection. \comm{If there exists a unique word $w$ of minimal length representing $g$, then we say $w$ is a \emph{unique geodesic}. Otherwise, $w$ is a non-unique geodesic. }

\begin{defn}\label{defn: length pres}
Let $G = \langle X \rangle$ be a group. We say $\phi \in \mathrm{Aut}(G)$ is:
\begin{enumerate}
    \item[(i)] \emph{length-preserving} if $|\pi(w)| = |\phi(\pi(w))|$ for any word $w \in X^{\ast}$, and
    \item[(i)] \emph{non-length preserving} otherwise.
\end{enumerate}
\end{defn}
For a group $G = \langle X \rangle$ and words $u, v \in X^{\ast}$, we use $u = v$ to denote equality of words, and $u =_{G} v$ to denote equality of the group elements represented by $u$ and $v$. For a word $w = x_{1}\dots x_{n} \in X^{\ast}$, where $x_{i} \in X$ for all $1 \leq i \leq n$, we define a \emph{cyclic permutation} of $w$ to be any word of the form
\[ w' = x_{i+1}\dots x_{n}x_{1}\dots x_{i},
\]
for some $i \in \{1, \dots, n-1 \}$.  

\subsection{Groups based on graphs}
We let $\Gamma$ denote a graph with vertex set $V(\Gamma)$ and edge set $E(\Gamma)$. We say $\Gamma$ is a finite simple graph if $V(\Gamma)$ and $E(\Gamma)$ are finite, and there are no loops or multiple edges. For any vertex $v \in V(\Gamma)$, we define the \emph{link} of a vertex $\mathrm{Lk}(v)$ as
\[ \mathrm{Lk}(v) = \{ x \in V(\Gamma) \mid \{v,x\} \in E(\Gamma) \}.
\]
Similarly we define the \emph{star} of a vertex $\mathrm{St}(v)$ as $\mathrm{St}(v) = \mathrm{Lk}(v) \cup \{v\}.$
\begin{defn}\label{defn:graph product}
    Let $\Gamma$ be a finite simple graph with vertices $V(\Gamma)$ and edges $E(\Gamma)$. Let $\mathcal{G} = \{G_{i} \mid i \in V(\Gamma) \} $ be a collection of groups which label the vertices of $\Gamma$. The associated \emph{graph product}, denoted $G_{\Gamma}$, is the group defined as the quotient
\[ \left( \ast_{i \in V(\Gamma)} G_{i}\right) / \langle \langle st= ts, \; s \in G_{i}, \; t \in G_{j}, \; \{i,j\} \in E(\Gamma) \rangle \rangle.
\]
If each vertex group is isomorphic to $\Z$, then $G_{\Gamma}$ is a \emph{right-angled Artin group} (RAAG), with presentation
    \[ \RAAG \cong \langle V(\Gamma) \mid [s,t] = 1 \; \text{for all} \; \{s,t\} \in E(\Gamma) \rangle.
    \]
If each vertex group is isomorphic to $\Z/2\Z$, then $G_{\Gamma}$ is a \emph{right-angled Coxeter group} (RACG), with presentation
    \[ W_{\Gamma} \cong \langle V(\Gamma) \mid s^{2} = 1 \; \text{for all} \; s \in V(\Gamma), [s,t] = 1 \; \text{for all} \; \{s,t\} \in E(\Gamma) \rangle.
    \] 
    We let $V(\Gamma) = \{s_{1}, \dots, s_{r} \}$ denote the standard generating set for $A_{\Gamma}$ or $W_{\Gamma}$. With this convention, standard generators commute in a RAAG or RACG if and only if an edge exists between vertices in our defining graph. This for example differs from the convention used in \cite{crisp_conjugacy_2009}. We let $X = V(\Gamma)^{\pm}$, and we order $X$ by setting $s_{1} < s^{-1}_{1} < \dots < s_{r} < s^{-1}_{r}.$ 
\end{defn}
  
\comm{
\begin{defn}
Let $\Gamma$ be a finite simple graph. The right-angled Artin group (RAAG) defined by $\Gamma$, denoted $A_{\Gamma}$, is the group with presentation
\[ \RAAG = \langle V(\Gamma) \mid \{s,t\} = 1 \; \text{for all} \; \{s,t\} \in E(\Gamma) \rangle
\]
Right-angled Coxeter groups (RACG) are defined similarly to RAAGs with the extra relations that every generator is of order 2. We denote $W_{\Gamma}$ to be the RACG defined by $\Gamma$.
\[ W_{\Gamma} = \langle V(\Gamma) \mid, s^{2} = 1 \; \text{for all} \; s \in V(\Gamma), \{s,t\} = 1 \; \text{for all} \; \{s,t\} \in E(\Gamma) \rangle
\]
\end{defn}
These groups are widely studied in geometric group theory, see \cite{Charney2007a} for further information.}

\comm{
\begin{defn}
A \emph{short exact sequence} of groups $N, G, H$ is a sequence of group homomorphisms
\[ 1 \rightarrow N \xrightarrow{i} G \xrightarrow{\pi} H \rightarrow 1
\]
such that
\begin{enumerate}
    \item $im(i) = ker(\pi)$
    \item $i$ is injective, i.e. $N \trianglelefteq G$
    \item $\pi$ is surjective, i.e. $H \cong G/N$
\end{enumerate}
A short exact sequence \emph{splits} if there exists a homomorphism $\gamma: H \rightarrow G$ such that $\pi \circ \gamma$ is the identity on $H$. 
\end{defn}

\begin{thm}\label{thm:SES}
A group $G$ is the semi-direct product of two groups $N$ and $H$ if and only if there exists a short exact sequence
\[ 1 \rightarrow N \xrightarrow{i} G \xrightarrow{\pi} H \rightarrow 1
\]
which splits. 
\end{thm}}

For RAAGs and RACGs, generators of the automorphism groups were classified by Servatius and Laurence \cite{laurence_generating_1995, servatius_automorphisms_1989}, and consist of four types:
\begin{enumerate}
    \item \textbf{Inversions}: For some $x \in V(\Gamma)$, send $x \mapsto x^{-1}$, and fix all other vertices.
    \item \textbf{Graph automorphisms}: Induced by the isomorphisms of the defining graph $\Gamma$.
    \item \textbf{Partial conjugations}: Let $x \in V(\Gamma)$, let $Q \subseteq \Gamma$ be the subgraph obtained by deleting $x$, all vertices adjacent to $x$ and all incident edges. Let $P \subseteq Q$ be a union of connected components of $Q$. A partial conjugation maps $p \mapsto xpx^{-1}$ for all $p \in P$, and fixes all other vertices.
    \item \textbf{Dominating transvections}: Let $x,y \in V(\Gamma)$, $x \neq y$, and assume $y$ is adjacent to all vertices in $\Gamma$ which are adjacent to $x$. A dominating transvection maps $x \mapsto xy^{\pm 1}$ or $x \mapsto y^{\pm 1}x$, and fixes all other vertices.
\end{enumerate}
\comm{
\section*{Automorphisms of RAAGs}
For any vertex $v \in V(\Gamma)$, we define the link of a vertex $Lk(v)$ as
\[ Lk(v) = \{ x \in V(\Gamma) \mid \{v,x\} \in E(\Gamma) \}
\]
Similarly we define the star of a vertex $St(v)$ as 
\[ St(v) = Lk(v) \cup \{v\}
\]
The automorphism group of a RAAG is generated by 4 types of automorphism:
\begin{enumerate}
    \item \textbf{Inversions:} $\phi: x \mapsto x^{-1}$
    \item \textbf{Graph automorphism:} Induced by any graph automorphism of $\Gamma$.
    \item \textbf{Dominated Transvections:} $x, y \in \Gamma$, $x \neq y$, $Lk(x) \subseteq St(y)$:
    \[\phi: x \mapsto xy^{\pm 1} \quad \text{or} \quad x \mapsto y^{\pm 1}x
    \]
    \item \textbf{Locally inner automorphism:} $x \in \Gamma$, $P \subseteq \Gamma \setminus St(x)$ is a union of connected components:
    \[ \phi: y \mapsto xyx^{-1} \quad \text{for all} \; y \in P
    \]
\end{enumerate}
This was conjectured by Servatius \cite{servatius_automorphisms_1989} in 1989, and proven for all possible RAAGs by Laurence six years later \cite{laurence_generating_1995}. We say $\phi \in \mathrm{Aut}(\RAAG)$ is:
\begin{itemize}
    \item \emph{length-preserving} if $l(w) = l(\phi(w))$ for any word $w \in X^{\ast}$
    \item \emph{non-length preserving} otherwise.
\end{itemize}
For RAAGs, it is clear that inversions and graph automorphisms are length-preserving, while dominated transvections and locally inner automorphisms are non-length preserving.} 
The following is straightforward to show from this classification.
\begin{lemma}\label{prop:combo len p}
Let $\phi \in \mathrm{Aut}(\RAAG)$ be a length-preserving automorphism. Then $\phi$ is a finite composition of inversions and graph automorphisms.
\end{lemma}
\comm{
\begin{proof}
Let $X = V(\Gamma)^{\pm}$. By \cref{defn: length pres}, $|\phi(x)| = 1$ for all $x \in X$, and so $\phi$ must permute $X$. This immediately implies that $\phi$ has finite order. For RAAGs, the automorphisms which permute $X$ are precisely inversions and graph automorphisms, which completes the proof.
\end{proof}}
\begin{defn}\cite[Section 2.2.1]{crisp_conjugacy_2009}\label{defn:CR RAAGS}
Let $\RAAG = \langle V(\Gamma) \rangle$, and let $v \in X^{\ast}$ be a geodesic. We say $v$ is \emph{cyclically reduced} if there does not exist a sequence of cyclic permutations, commutation relations and free reductions to a geodesic word $w \in X^{\ast}$, such that $l(v) > l(w)$.
\end{defn}
In particular, a geodesic $w \in X^{\ast}$ is cyclically reduced if and only if $w$ does not have the form $w = x_{1}s^{\pm 1}_{i}x_{2}s^{\mp 1}_{i}x_{3}$, where $x_{1}, x_{2}, x_{3} \in X^{\ast}$, $s_{i} \in V(\Gamma)$ and all letters of $x_{1}$ and $x_{3}$ commute with $s_{i}$ in $\RAAG$. 
\subsection{Decision problems}
We recall two of Dehn's decision problems in group theory.
\begin{defn}
    Let $G$ be a group given by a finite presentation with generating set $X$.
    \begin{enumerate}
        \item The \emph{word problem} for $G$, denoted $\mathrm{WP}(G)$, takes as input a word $w \in X^{\ast}$, and decides whether it represents the trivial element of $G$.
        \item The \emph{conjugacy problem} for $G$, denoted $\mathrm{CP}(G)$, takes as input two words $u,v \in X^{\ast}$, and decides whether they represent conjugate elements in $G$. We write $u \sim v$ when two words $u,v \in X^{\ast}$ represent conjugate elements.
    \end{enumerate}
\end{defn}
\comm{
\begin{exmp}\label{exmp: groups solvable CP}
We provide a non-exhaustive list of groups with solvable conjugacy problem, and the time complexity, if known.
\begin{enumerate}
    \item \cite{epstein_linearity_2006} Hyperbolic: linear time. 
    \item \cite[III.$\Gamma$, Theorem 1.12]{bridson_metric_1999} $\mathrm{CAT}(0)$ groups: exponential time.
    \item \cite[Theorem 3.24]{green_graph_1990} Graph products: solvable conjugacy problem if and only if each vertex group has solvable conjugacy problem.
    \item \cite{crisp_conjugacy_2009} Right angled Artin/Coxeter groups: linear time.
    \comm{
    \item\label{thm:conj prob decid DAm}\cite[Proposition 3.1]{Holt2015}
    Dihedral Artin groups: linear time.}
\end{enumerate}
\end{exmp}}

\begin{defn}\label{defn:twisted conjugacy}
    Let $G = \langle X \rangle$, let $u,v \in X^{\ast}$, and let $\phi \in \mathrm{Aut}(G)$.
    \begin{enumerate}
        \item We say $u$ and $v$ are \emph{$\phi$-twisted conjugate}, denoted $u \sim_{\phi} v$, if there exists an element $w \in G$ such that $v =_{G} \phi(w)^{-1}uw$.
        \item The \emph{$\phi$-twisted conjugacy problem} for $G$, denoted $\mathrm{TCP}_{\phi}(G)$, takes as input two words $u,v \in X^{\ast}$, and decides whether they represent groups elements which are $\phi$-twisted conjugate to each other in $G$.
        \item The (uniform) \emph{twisted conjugacy problem} for $G$, denoted $\mathrm{TCP}(G)$, takes as input two words $u,v \in X^{\ast}$ and $\phi \in \mathrm{Aut}(G)$, and decides whether $u$ and $v$ represent group elements which are $\phi$-twisted conjugate in $G$.
    \end{enumerate}
\end{defn}
A solution to the $\mathrm{TCP}(G)$ implies a solution to $\mathrm{TCP}_{\phi}(G)$ for all $\phi \in \mathrm{Aut}(G)$, and therefore a solution to the $\mathrm{CP}(G)$ and the $\mathrm{WP}(G)$. 
\subsection{Group extensions}
We assume the reader is familiar with quasi-isometries and geometric group actions, and refer the reader to \cite{margalit_office_2017} for details.
\begin{defn}
    An \emph{extension} of a group $H$ by $N$ is a group $G$ such that $N \trianglelefteq G$ and $G/N \cong H$. This can be encoded by a short exact sequence of groups
\[ 1 \rightarrow N \rightarrow G \rightarrow H \rightarrow 1.
\]
A group extension is \emph{split} if and only if $G$ is a semi-direct product, that is $G \cong N \rtimes_{\alpha} H$ for some homomorphism $\alpha \colon H \rightarrow \mathrm{Aut}(N)$. Unless otherwise stated, we assume $H$ is a finite cyclic group, and so $N$ has finite index in $G$. It is sufficient to specify $\alpha(t)$ for a generator $t$ of $H$. We let $\phi = \alpha(t)$, and so $G$ has presentation
\begin{equation}\label{eqn:original split}
    G \cong N \rtimes_{\phi} H = \langle X \cup t \mid R(N), \; t^{m} = 1, \; t^{-1}xt = \phi(x) \; (x \in X) \; \rangle,
\end{equation}
where $N$ has presentation $\langle X \mid R(N) \rangle$, and $\phi \in \mathrm{Aut}(N)$ is of finite order, such that $m$ divides the order of $\phi$. Any element of $G$ can be written in the form $g = vt^{l}$, where $v \in N$ and  $0 \leq l \leq m-1$. Unless otherwise stated, we will assume a finite group extension is that of the form in \cref{eqn:original split}, where the order of $\phi \in \mathrm{Aut}(N)$ is equal to the order of $t \in H$.
\end{defn}
\begin{defn}
    Let $\mathcal{P}$ be a group property. We say a group $G$ is \emph{virtually $\mathcal{P}$} if there exists a finite index subgroup $H \leq G$ such that $H$ has property $\mathcal{P}$.
\end{defn}
In particular, groups $G$ of the form in \cref{eqn:original split} are virtual $N$ groups, and $G$ and $N$ are quasi-isometric by the Milnor-\v{S}varc Lemma. There is an immediate link between conjugacy in a group extension $G \cong N \rtimes_{\phi} H$, and twisted conjugacy in the normal subgroup $N$. 
\comm{
\begin{defn}\cite[Chapter 7]{margalit_office_2017}
    Let $(X, d_{X})$ and $(Y, d_{Y})$ be metric spaces. A function $f \colon X \rightarrow Y$ is a \emph{quasi-isometric embedding} if there exist constants $K \geq 1$ and $C \geq 0$ such that
    \[ \frac{1}{k}d_{X}(x_{1}, x_{2}) - C \leq d_{Y}(f(x_{1}), f(x_{2})) \leq Kd_{X}(x_{1}, x_{2}) + C,
    \]
    for all $x_{1}, x_{2} \in X$. A quasi-isometric embedding is called a \emph{quasi-isometry} if the function $f \colon X \rightarrow Y$ is \emph{quasi-dense}, that is, there exists $D > 0$ such that for every $y \in Y$, there exists $x \in X$ such that $d_{Y}(f(x), y) \leq D$. Two metric spaces $(X, d_{X})$ and $(Y, d_{Y})$ are \emph{quasi-isometric} if there exists a quasi-isometry $f \colon X \rightarrow Y$. 
    
    Two groups $G$ and $H$ are quasi-isometric if there exists finite generating sets $S$ and $T$ for $G$ and $H$ respectively, such that the Cayley graphs $\mathrm{Cay}(G,S)$ and $\mathrm{Cay}(H,T)$ are quasi-isometric. 
\end{defn}
For any metric space $(X, d)$, we let $B(x,r)$ denote the ball of radius $r>0$ about $x \in X$. 
\begin{defn}\cite{margalit_office_2017}
    Let $G$ be a group which acts on a metric space $(X,d)$ by isometries. This group action is
    \begin{enumerate}
        \item[(i)] \emph{faithful} if its kernel is trivial,
        \item[(ii)] \emph{cobounded} if for all $x_{0} \in X$, there exists $r > 0$ such that for all $x \in X$, there exists $g \in G$ such that $x \in B(gx_{0}, r)$, and
        \item[(iii)] \emph{properly discontinuous} if for all $x \in X$, $r>0$, there exists finitely many $g \in G$ such that $B(gx, r) \cap B(x,r) \neq \varnothing$.
    \end{enumerate}
    A group action is \emph{geometric} if it is cocompact and properly discontinuous.
\end{defn}}
\comm{
\begin{lemma}[Milnor-\v{S}varc Lemma]\label{lemma:milnor-swarc}
    Let $G$ be a group and let $(X,d)$ be a proper geodesic metric space. Suppose $G$ acts geometrically on $X$. Then $G$ is finitely generated and quasi-isometric to $X$. 
\end{lemma}}
\comm{
\begin{cor}
    Let $G,H$ be groups such that $H \leq G$ is of finite index and $G$ is finitely generated. Then $G$ is quasi-isometric to $H$. 
\end{cor}}

\comm{
\begin{rmk}\label{rmk:CP not quasi isometry invariant}
    While there exist many properties, such as hyperbolicity and standard growth, which are quasi-isometry invariant, there are also important properties which are not. A relevant example is given by Collins and Miller \cite{collins_conjugacy_1977}, who constructed groups $H \leq G$, where $H$ is of finite index in $G$, such that $H$ has solvable conjugacy problem but $G$ does not. In particular, although RAAGs have solvable conjugacy problem \cite{crisp_conjugacy_2009}, we cannot assume that virtual RAAGs will also have solvable conjugacy problem.
\end{rmk}}

\comm{
\begin{rmk}
    For notation, we let $\widehat{X} = X \cup \{t\}^{\pm}$ denote the generating set for the group extension $G$ as in \cref{eqn:original split}. The order we take on $\widehat{X}$ is given by $x < t < t^{-1}$ for all $x \in X$. 
\end{rmk}}

\comm{
\begin{lemma}(\cite{bogopolski_conjugacy_2006}, Page 4)\label{lem:Ventura}
Let $G$ be a group extension of the form \eqref{eqn:original split}. Let $ut^{a}, vt^{b} \in G$ be two elements in $G$, with $u,v \in N$, $0 \leq a,b \leq m-1$. Then $ut^{a} \sim vt^{b}$ if and only if $a=b$ and $v \sim_{\phi^{a}} \phi^{k}(u) $, for some integer $k$ where $0 \leq k \leq m-1$. 
\end{lemma}}

\begin{prop}\cite[Proposition 4.7]{bogopolski_orbit_2009}\label{cor:CP extensions iff TCP in base}
    Let $G$ be a group extension as in \eqref{eqn:original split}. Then $\mathrm{CP}(G)$ is solvable if and only if $\mathrm{TCP}_{\phi^{a}}(N)$ is solvable for all $0 \leq a \leq m-1$. 
\end{prop}
 
\comm{
We highlight the main result from that paper, which provides new examples of groups with solvable or unsolvable conjugacy problem, by studying twisted conjugacy and orbit decidability. 

\begin{thm}\label{thm:orbit conditions}\cite[Theorem 3.1]{bogopolski_orbit_2009}
    Let 
    \[ 1 \rightarrow F \xrightarrow{\alpha} G \xrightarrow{\beta} H \rightarrow 1
    \]
    be an algorithmic short exact sequence of groups such that
    \begin{enumerate}
        \item[(i)] $F$ has solvable twisted conjugacy problem,
        \item[(ii)] $H$ has solvable conjugacy problem, and
        \item[(iii)] for every $1 \neq h \in H$, the subgroup $\langle h \rangle$ has finite index in its centralizer $C_{H}(h)$, and there is an algorithm which computes a finite set of coset representatives $z_{h,1}, \dots z_{h, t_{h}} \in H$, i.e.
        \[ C_{H}(h) = \langle h \rangle z_{h,1} \sqcup \dots \sqcup \langle h \rangle z_{h,t_{h}}.
        \]
        \end{enumerate}
        Then the conjugacy problem for $G$ is decidable if and only if the action subgroup $A_{G} = \{\varphi_{g} \mid g \in G \} \leq \mathrm{Aut}(F)$ is orbit decidable 
\end{thm}
Note we assume in this result that groups are described in an `algorithmic way' (see \cite[Section 2]{bogopolski_orbit_2009}).}

\subsection{Conjugacy growth}\label{sec:conjugacy growth}
\begin{defn}
A formal power series $f(z) = \sum^{\infty}_{i=0} a_{i}z^{i} \in \Z[[z]]$ is \emph{rational} if there exist non-zero polynomials $p(z), q(z) \in \Z[z]$ such that $f(z) = \frac{p(z)}{q(z)}$. Equivalently, $f(z)$ is rational if the coefficients $a_{i}$ of $f(z)$ satisfy a linear recursion.  Furthermore, $f(z)$ is \emph{irrational} if it is not rational. We say $f(z)$ is \emph{algebraic} if it is in the algebraic closure of $\Q(z)$, i.e.\ it is the solution to a polynomial equation with coefficients from $\Q(z)$. If $f(z)$ is not algebraic, then $f(z)$ is \emph{transcendental}.
\end{defn}
\begin{defn}\label{defn:conjugacy growth function series}
    For a group $G = \langle X \rangle$, we let $[g]_{c}$ denote the conjugacy class of $g \in G$. We define the \emph{length up to conjugacy} of an element $g \in G$ by
\[ |g|_{c} := \text{min}\{ |h| \mid h \in [g]_{c} \}.
\]
    The \emph{(strict) conjugacy growth function}, denoted as $c(n) = c_{G,X}(n)$, is defined as the number of conjugacy classes of length $=n$, i.e.\ 
\[ c(n) = \#\{[g]_{c} \mid |[g]_{c}| = n\}.\]
The \emph{conjugacy growth series} $C(z) = C_{G,X}(z)$ is defined to be the (ordinary) generating function of $c(n)$, i.e.
\[ C(z) = \sum\limits_{n = 0}^{\infty} c(n) z^n. \]
\end{defn}
The nature of this series has been determined for various classes of groups \cite{antolin_formal_2017, ciobanu_conjugacy_2024, ciobanu_conjugacy_2020, Guba2010}. All known results support the conjecture that the only finitely presented groups with rational conjugacy growth series are virtually abelian groups \cite[Conjecture 7.2]{ciobanu_conjugacy_2020}.

\section{Examples of extensions of RAAGs and RACGs}\label{sec:survey}

\subsection{Hyperbolic, virtually free and free abelian extensions}\label{sec:free abelian TCP}
Hyperbolicity is a quasi-isometry invariant, and so determining when RAAGs and RACGs are hyperbolic provides a positive solution to the conjugacy problem in finite extensions as in \cref{eqn:original split} \cite{epstein_linearity_2006}. For RACGs, there exists a graph criterion for these groups to be hyperbolic.
\begin{thm}\label{RACG:hyperbolic}\cite[Theorem 17.1]{moussong_hyperbolic_1988}
A RACG $W_{\Gamma}$ is hyperbolic if and only if $\Gamma$ does not contain any induced squares. 
\end{thm}
For graph products, we have the following criteria on our defining graph to obtain hyperbolicity.
\begin{thm}\cite[Theorem 5.1]{holt_generalising_2012}
Let $G_{\Gamma}$ be a graph product. If $G_{\Gamma}$ is hyperbolic, then no two vertices with infinite vertex groups are adjacent.
\end{thm}
For RAAGs, this immediately leads to the following result. 

\begin{cor}
A RAAG $A_{\Gamma}$ is hyperbolic if and only if $A_{\Gamma}$ is isomorphic to a free group.
\end{cor}
We now consider the two extremes of RAAGs, namely free groups and free abelian groups. Virtually free groups are hyperbolic, and so have solvable conjugacy problem. Alternatively, one can apply the following result along with \cite[Proposition 4.7]{bogopolski_orbit_2009} to solve the conjugacy problem in virtually free groups.

\begin{thm}\label{thm:TCP free}\cite[Theorem 1.5]{bogopolski_conjugacy_2006}
    Let $F$ be a finitely generated free group. Then the twisted conjugacy problem is solvable in $F$.    
\end{thm}
\comm{
The proof relies on the following result of Maslakova. 

\begin{thm}\cite[Theorem 1.1]{bogopolski_algorithm_2016}
    There exists an algorithm to compute a finite generating set for the fixed point subgroup of an arbitrary automorphism of a free group of finite rank.
\end{thm}
We note here an updated reference due to an incomplete proof in the original result (see the following note: \url{https://enric-ventura.staff.upc.edu/ventura/files/13comment.pdf}). }
\comm{
Whilst there are recent results which show that fixed point subgroups are finitely generated in RAAGs for untwisted automorphisms \cite{fioravanti_coarse-median_2023}, it remains open whether there exists an algorithm to compute such a finite generating set, and so currently we cannot use a similar technique to prove the twisted conjugacy problem in RAAGs.}

The opposite extreme for RAAGs also has solvable twisted conjugacy problem. Indeed, if a RAAG is free abelian of rank $n$, then the twisted conjugacy problem reduces to solving a system of $n$ linear equations with $n$ unknowns. 

We also mention the extreme cases for RACGs. If our defining graph $\Gamma$ is a complete graph, then our RACG $W_{\Gamma}$ is a finite abelian group, and so has solvable twisted conjugacy problem. On the other hand, if $\Gamma$ is completely disconnected, then $W_{\Gamma}$ is a free product of finite cyclic groups, which is virtually free. Therefore any finite extension is also virtually free, and so is hyperbolic with solvable conjugacy problem. 

\subsection{CAT(0) groups}\label{sec:CAT0}
We refer the reader to \cite{bridson_metric_1999} for details on $\mathrm{CAT}(0)$ spaces and groups. Recall a group action is \emph{geometric} if it is cocompact and properly discontinuous.
\comm{
\begin{defn}\cite{bridson_metric_1999}\label{cat0defn}
Let $(Y,d)$ be a geodesic metric space. Then $(Y,d)$ is a CAT(0) space if the following holds (see \cref{fig:CAT0 triangles}):

Let $a,b,c \in Y$, let $\Delta$ be the geodesic triangle between these 3 points. Let $\overline{\Delta}$ be a comparison triangle for $\Delta$, that is, $\overline{\Delta} \in \mathbb{E}^{2}$ such that $d_{Y}(a,b) = d_{\mathbb{E}^{2}}\left(\overline{a}, \overline{b}\right),$ $d_{Y}(a,c) = d_{\mathbb{E}^{2}}\left(\overline{a}, \overline{c}\right)$, and $d_{Y}(b,c) = d_{\mathbb{E}^{2}}\left(\overline{b}, \overline{c}\right)$. Then for any $x \in [a,b], \; \overline{x} \in \left[\overline{a}, \overline{b}\right]$ such that $d_{Y}(a,x) = d_{\mathbb{E}^{2}}\left(\overline{a}, \overline{x}\right)$, and $y \in [a,c], \; \overline{y} \in [\overline{a}, \overline{c}]$ such that $d_{Y}(a,y) = d_{\mathbb{E}^{2}}(\overline{a}, \overline{y})$, we have that
\[ d_{Y}(x,y) \leq d_{\mathbb{E}^{2}}(\overline{x}, \overline{y}).
\]
\begin{figure}[h]
    \centering

\tikzset{every picture/.style={line width=0.75pt}} 

\begin{tikzpicture}[x=0.75pt,y=0.75pt,yscale=-1,xscale=1]

\draw   (399.01,81.16) -- (411.78,200.42) -- (293.89,141.25) -- cycle ;
\draw    (110,143.25) .. controls (158.5,126.75) and (180.67,113.83) .. (201,80.83) ;
\draw    (110,143.25) .. controls (169.67,152.17) and (192.6,159) .. (240,200.25) ;
\draw    (201,80.83) .. controls (205.33,126.83) and (216.33,159.5) .. (240,200.25) ;
\draw  [fill={rgb, 255:red, 0; green, 0; blue, 0 }  ,fill opacity=1 ] (143.33,129.63) .. controls (143.33,128.87) and (143.95,128.25) .. (144.71,128.25) .. controls (145.47,128.25) and (146.08,128.87) .. (146.08,129.63) .. controls (146.08,130.38) and (145.47,131) .. (144.71,131) .. controls (143.95,131) and (143.33,130.38) .. (143.33,129.63) -- cycle ;
\draw  [fill={rgb, 255:red, 0; green, 0; blue, 0 }  ,fill opacity=1 ] (150.67,150.96) .. controls (150.67,150.2) and (151.28,149.58) .. (152.04,149.58) .. controls (152.8,149.58) and (153.42,150.2) .. (153.42,150.96) .. controls (153.42,151.72) and (152.8,152.33) .. (152.04,152.33) .. controls (151.28,152.33) and (150.67,151.72) .. (150.67,150.96) -- cycle ;
\draw  [fill={rgb, 255:red, 0; green, 0; blue, 0 }  ,fill opacity=1 ] (330.83,160.13) .. controls (330.83,159.37) and (331.45,158.75) .. (332.21,158.75) .. controls (332.97,158.75) and (333.58,159.37) .. (333.58,160.13) .. controls (333.58,160.88) and (332.97,161.5) .. (332.21,161.5) .. controls (331.45,161.5) and (330.83,160.88) .. (330.83,160.13) -- cycle ;
\draw  [fill={rgb, 255:red, 0; green, 0; blue, 0 }  ,fill opacity=1 ] (324.83,122.79) .. controls (324.83,122.03) and (325.45,121.42) .. (326.21,121.42) .. controls (326.97,121.42) and (327.58,122.03) .. (327.58,122.79) .. controls (327.58,123.55) and (326.97,124.17) .. (326.21,124.17) .. controls (325.45,124.17) and (324.83,123.55) .. (324.83,122.79) -- cycle ;
\draw    (144.71,128.25) -- (152.04,150.96) ;
\draw    (326.21,122.79) -- (332.21,160.13) ;
\draw    (177.6,163.4) -- (182.2,156.2) ;
\draw    (358.8,178.2) -- (363.4,171) ;
\draw    (160.2,116.37) -- (165.8,122.77) ;
\draw    (157.8,118.2) -- (163.8,124.6) ;
\draw    (342.8,108.2) -- (348.6,115) ;
\draw    (345.14,106.2) -- (349.31,111.38) -- (350.94,113.4) ;
\draw    (204.4,129.8) -- (214.2,127.8) ;
\draw    (206,135) -- (215.8,133) ;
\draw    (205.2,132.6) -- (215,130.6) ;
\draw    (400.8,139.4) -- (410.6,137.4) ;
\draw    (400.4,137) -- (410.2,135) ;
\draw    (400,134.6) -- (409.8,132.6) ;

\draw (143.87,215.27) node [anchor=north west][inner sep=0.75pt]   [align=left] {(a) $\displaystyle \Delta \in \ Y$};
\draw (319.97,214) node [anchor=north west][inner sep=0.75pt]   [align=left] {(b) $\displaystyle \overline{\Delta} \in \ \mathbb{E}^{2}$};
\draw (96.93,140.13) node [anchor=north west][inner sep=0.75pt]  [font=\footnotesize] [align=left] {$\displaystyle a$};
\draw (239.87,194.33) node [anchor=north west][inner sep=0.75pt]  [font=\footnotesize] [align=left] {$\displaystyle b$};
\draw (197.4,65.4) node [anchor=north west][inner sep=0.75pt]  [font=\footnotesize] [align=left] {$\displaystyle c$};
\draw (283.33,133.73) node [anchor=north west][inner sep=0.75pt]  [font=\footnotesize] [align=left] {$\displaystyle \overline{a}$};
\draw (411.67,192.47) node [anchor=north west][inner sep=0.75pt]  [font=\footnotesize] [align=left] {$\displaystyle \overline{b}$};
\draw (394.33,66.6) node [anchor=north west][inner sep=0.75pt]  [font=\footnotesize] [align=left] {$\displaystyle \overline{c}$};
\draw (146.4,155.73) node [anchor=north west][inner sep=0.75pt]  [font=\footnotesize] [align=left] {$\displaystyle x$};
\draw (137.33,112.93) node [anchor=north west][inner sep=0.75pt]  [font=\footnotesize] [align=left] {$\displaystyle y$};
\draw (323.33,160.6) node [anchor=north west][inner sep=0.75pt]  [font=\footnotesize] [align=left] {$\displaystyle \overline{x}$};
\draw (319,107.2) node [anchor=north west][inner sep=0.75pt]  [font=\footnotesize] [align=left] {$\displaystyle \overline{y}$};

\end{tikzpicture}

    \caption{$\mathrm{CAT}(0)$ inequality}
    \label{fig:CAT0 triangles}
\end{figure}

\end{defn}
\comm{
\begin{figure}[h]
\begin{subfigure}{0.5\textwidth}
\centering
\begin{tikzpicture}[xscale=4,yscale=3,rotate=45]
	\draw [black,postaction={decoration={markings,mark=at position 0.5 with {\arrow{|}}},decorate}] (0,0) coordinate[label=below:$b$](B) arc[start angle=0,end angle=60,radius=1];
	\draw [black,postaction={decoration={markings,mark=at position 0.49 with {\arrow{|}}},decorate},postaction={decoration={markings,mark=at position 0.51 with {\arrow{|}}},decorate}] ({cos(60)},{sin(60)}) coordinate[label=right:$c$](C) arc[start angle=300,end angle=240,radius=1] coordinate[label=left:$a$](A);
	\draw [black,postaction={decoration={markings,mark=at position 0.48 with {\arrow{|}}},decorate},postaction={decoration={markings,mark=at position 0.5 with {\arrow{|}}},decorate},postaction={decoration={markings,mark=at position 0.52 with {\arrow{|}}},decorate}] ({cos(60)},{sin(60)}) arc[start angle=120,end angle=180,radius=1];
    \draw [draw=none] (0,0) arc[start angle=0,end angle=40,radius=1] coordinate[circle, fill, inner sep=1pt,label=below:$x$](X);
    \draw [draw=none] ({cos(60)},{sin(60)}) arc[start angle=300,end angle=260,radius=1] coordinate[circle, fill, inner sep=1pt,label=above:$y$](Y);
    \draw [black] (X) -- (Y);
\end{tikzpicture}
\end{subfigure}
~
\begin{subfigure}{0.5\textwidth}
\centering
\begin{tikzpicture}[xscale=4,yscale=3,rotate=45]
	\draw [black,postaction={decoration={markings,mark=at position 0.5 with {\arrow{|}}},decorate}] ({-cos(60)},{sin(60)}) coordinate[label=left:$\overline{a}$](A) --(0,0) coordinate[label=below:$\overline{b}$](B);
	\draw [black,postaction={decoration={markings,mark=at position 0.48 with {\arrow{|}}},decorate},postaction={decoration={markings,mark=at position 0.5 with {\arrow{|}}},decorate},postaction={decoration={markings,mark=at position 0.52 with {\arrow{|}}},decorate}] (B) -- ({cos(60)},{sin(60)}) coordinate[label=right:$\overline{c}$](C);
	\draw [black,postaction={decoration={markings,mark=at position 0.49 with {\arrow{|}}},decorate},postaction={decoration={markings,mark=at position 0.51 with {\arrow{|}}},decorate}] (C) -- (A);
	\coordinate [circle, fill, inner sep=1pt,label=below:$\overline{x}$](X) at ({-2/3*cos(60)},{2/3*sin(60)});
	\coordinate [circle, fill, inner sep=1pt,label=above:$\overline{y}$](Y) at ({-1/3*cos(60)},{sin(60)});
	\draw [black] (X) -- (Y);
\end{tikzpicture}
\end{subfigure}
~
\begin{minipage}{.5\textwidth}
	\leavevmode\subcaption{$\Delta\in X$}
\end{minipage}
\begin{minipage}{.5\textwidth}
	\leavevmode\subcaption{$\overline{\Delta}\in\mathbb{E}^2$}
\end{minipage}
\end{figure}}}

\begin{defn}\label{def:CAT(0) group}
    Let $G$ be a finitely generated group. We say $G$ is a \emph{$\mathrm{CAT}(0)$ group} if $G$ acts faithfully and geometrically on a CAT(0) metric space. 
\end{defn}
Both RAAGs and RACGs are examples of $\mathrm{CAT}(0)$ groups. We have the following motivating result for finite extensions of right-angled Coxeter groups.
\begin{thm}\label{thm:RACGs CAT}\cite[Theorem 1]{cunningham_rm_2016}
    Let $W_{\Gamma}$ be a RACG and $\phi \in \mathrm{Aut}(W_{\Gamma})$ be either a graph automorphism, partial conjugation, or a transvection. Let $m < \infty$ denote the order of $\phi$. Then the finite extension $G = W \rtimes_{\phi} \Z / m\Z$ is a $\mathrm{CAT}(0)$ group. 
\end{thm}
We can extend this result to RAAGs for all length-preserving automorphisms.

\CATzero

The following is then immediate since $\mathrm{CAT}(0)$ groups have solvable conjugacy problem \cite[III.$\Gamma$, Theorem 1.12]{bridson_metric_1999}.

\CATzeroresult
To prove \cref{thm: CAT(0) extension RAAG}, we refer the reader to \cite{schwer_lecture_2019} for information on $\mathrm{CAT}(0)$ cube complexes. In particular, a cube complex is $\mathrm{CAT}(0)$ if it is simply connected and all vertex links are flag.

\begin{proof}[Proof of \cref{thm: CAT(0) extension RAAG}]
    Let $S$ be the Salvetti complex of the RAAG defined by $\Gamma$. It is well known that $S$ is a $\mathrm{CAT}(0)$ cube complex, and that any RAAG $\RAAG$ acts geometrically on its corresponding Salvetti complex. Moreover, the 1-skeleton $S^{1}$ is precisely the Cayley graph $\mathrm{Cay}(\RAAG, X)$ (recall $X = V(\Gamma)^{\pm}$), where two vertices $v_{u}, v_{ug} \in S^{1}$ are adjacent if and only if $g \in X$. The aim is to extend this action to a geometric action of the extension $A_{\phi}$ on the complex $S$. We let $\Phi_{s} \colon v_{u} \mapsto v_{su}$ denote the left-multiplication action, for all $s \in \RAAG$, and let $m$ be the order of $\phi$. 
    \comm{
    First let $\phi$ be an inversion. By definition, for any group element $u \in \RAAG$, the vertices $v_{u}$ and $v_{gu}$ are connected by an edge in the 1-skeleton $X^{1}$ if and only if $g \in S$. Therefore the vertices $v_{\phi(u)}$ and $v_{\phi(gu)}$ are connected by an edge in $X^{1}$ if and only if $\phi(g) \in S$. However $\phi(g) \in \{ g, g^{-1} \}$, and if $g \in S$, then $g^{-1} \in S$. Hence the map $\Phi \colon v_{w} \mapsto v_{\phi(w)}$ preserves adjacency, and so $\Phi \in \mathrm{Isom}(X)$. \\

    We now consider graph automorphisms. Since graph automorphisms map bijectively edges to edges, we have that $v_{\phi(w)}, v_{\phi(u)}$ are connected by an edge in $X^{1}$ if and only if $\phi(w)^{-1}\phi(u) \in S$. Hence the map $\Phi \colon v_{w} \mapsto v_{\phi(w)}$ preserves adjacency in $X^{1}$, and so $\Phi \in \mathrm{Isom}(X)$.}
    
    We first show that the adjacency relation is preserved by the map $v_{w} \mapsto v_{\phi(w)}$. Note that since $\phi$ is length-preserving, then $g \in X$ if and only if $\phi(g) \in X$. Therefore vertices $v_{\phi(u)}$ and $v_{\phi(ug)}$ are joined by an edge in $S^{1}$ if and only if $\phi(g) \in X$ as required. Note by construction that any permutation of the vertex set $S^{1}$ which respects the adjacency relation determines an isometry of the complex $S$. Therefore the map $v_{w} \mapsto v_{\phi(w)}$ determines an isometry $\Phi \colon v_{w} \mapsto v_{\phi(w)}$ of $S$. We now define a map from the generators of $A_{\phi}$ to the isometry group of $S$ by the following rule:
\[ s_{i} \mapsto \Phi_{s_{i}} \; \text{for all} \; s_{i} \in X, \; t \mapsto \Phi.
\]
We check the relations from the presentation for $A_{\phi}$ given in \cref{semidirect1}:
\begin{align*}
     v_{w} &\mapsto v_{[s_{i},s_{j}]w} = v_{w} \; \text{for all} \; \{s_{i},s_{j}\} \in E(\Gamma) \Rightarrow [\Phi_{s_{i}}, \Phi_{s_{j}}] = 1,\\
     v_{w} &\mapsto v_{\phi^{m}(w)} = v_{w} \Rightarrow \Phi^{m} = 1, \\
     v_{w} &\mapsto v_{\phi(s_{i}\phi^{-1}(w))} = v_{\phi(s_{i}) \cdot \phi(\phi^{-1}(w))} = v_{\phi(s_{i})w} \Rightarrow \Phi \Phi_{s_{i}} \Phi^{-1} = \Phi_{\phi(s_{i})},
\end{align*}
and so the group relations are satisfied. Hence this rule determines an action $A_{\phi} \curvearrowright S$. We now want to show this action is a faithful geometric action. Note the action $\RAAG \curvearrowright S$ is a faithful geometric action, so it remains to check the following cases for $\Phi$. 

If $\phi(x) = x$ for all $x \in S^{1}$, then $\phi$ is trivial, which gives us a faithful action. Let $\Phi_{g}\Phi^{a} \in \mathrm{Isom}(S)$ (corresponding to the group element $gt^{a} \in A_{\phi}$). For all $w \in S^{1}$, we have
\[ \Phi_{g}\Phi^{a} \colon v_{w} \mapsto v_{g\phi^{a}(w)}
\]
where $g\phi^{a}(w) \in S^{1}$. Therefore if the action $A_{\phi} \curvearrowright S$ is not properly discontinuous, then the action $\RAAG \curvearrowright S$ must also not be properly discontinuous, which is a contradiction. Finally, cobounded follows since $\RAAG \leq A_{\phi}$, so for any $s \in S$, $r>0$, we can choose $g \in \RAAG \leq A_{\phi}$ such that $x \in B(gx_{0}, r)$. Hence the action $A_{\phi} \curvearrowright S$ is a faithful, geometric action by isometries, and so $A_{\phi}$ is a CAT(0) group.
\end{proof}
We can use this geometric property to study the conjugacy growth series of virtual RAAGs and RACGs (recall \cref{sec:conjugacy growth}). The following definitions can be found for example in \cite{gekhtman_counting_2022}.
\begin{defn}
    Let $(Y,d)$ be a geodesic metric space. Given $y \in Y$ and a subset $Z \subseteq Y$, let $\pi_{Z}(y)$ be the set of points $z \in Z$ such that $d(y,z) = d(y, Z)$. The \emph{projection} of a subset $A \subseteq Y$ to $Z$ is defined as 
    \[ \pi_{Z}(A) := \bigcup_{a \in A} \pi_{Z}(a).
    \]
    For a given $C \geq 1$, a subset $Z \subseteq Y$ is \emph{$C$-contracting} if for any geodesic $\gamma$, with $d(\gamma, Z) \geq C$, we have $d^{\pi}_{Z}(\gamma) \leq C$, where $d^{\pi}_{Z}(\gamma) := \mathrm{diam}(\pi_{Z}(\gamma))$ is the diameter of the projection of $\gamma$ to $X$.
\end{defn}

\begin{defn}\label{defn:contracting}
    Let $G$ be a group which acts properly by isometries on a geodesic metric space $(Y,d)$. An element $h \in G$ is \emph{contracting} if for some basepoint $o \in Y$, the orbit $\langle h \rangle\cdot o$ is $C$-contracting (for some $C \geq 1$), and the map 
    \begin{align*}
        \Phi \colon \Z &\rightarrow Y \\
        n &\mapsto h^{n}\cdot o
    \end{align*}
    is a quasi-isometric embedding. 
\end{defn}
Here a non-elementary group is one which is not virtually cyclic.

\begin{thm}\label{thm:contracting conj growth}\cite[Corollary 1.8]{gekhtman_counting_2022}
    Let $G$ be a non-elementary group with a finite generating set $X$. If $G$ has a contracting element, with respect to the action on $\mathrm{Cay}(G,X)$, then the conjugacy growth series is transcendental. 
\end{thm}

We now collect results from the literature to establish when finite extensions of RAAGs and RACGs have contracting elements, and so have transcendental conjugacy growth series. An isometry $g \in \mathrm{Isom}(X)$ is called $\emph{rank-1}$ if some axis of $g$ does not bound a half-plane. 

\begin{thm}\label{thm:rank 1 cat0 contract}\cite[Theorem 5.4]{bestvina_characterization_2009}
    Rank-1 elements in CAT(0) groups are contracting.
\end{thm}

\begin{thm}\cite[Theorem 5.2]{behrstock_divergence_2012}\label{prop:RAAG product contracting}
    If $\RAAG$ is not a direct product, then $\RAAG$ contains a rank-1 element.
\end{thm}
In particular, if $\RAAG$ is not a direct product, then $\RAAG$ contains a contracting element, with respect to the action on the Salvetti complex. This implies that the conjugacy growth series is transcendental for all irreducible RAAGs by \cref{thm:contracting conj growth}. When we consider the corresponding extension $A_{\phi}$, as in \cref{semidirect1}, we find that this group also contains a contracting element, which again by \cref{thm:contracting conj growth} implies that the conjugacy growth of $A_{\phi}$ is transcendental.
\comm{
If $X$ is a proper $\mathrm{CAT}(0)$ space, then $g \in \mathrm{Isom}(X)$ is a contracting element, with respect to the $\mathrm{CAT}(0)$ metric, if and only if $g$ is a rank-1 element. }

\comm{
To use \cref{thm:contracting conj growth}, we require the following result.

\begin{prop}
    Let $A_{\phi}$ be a virtual RAAG as in \cref{semidirect1}. Then there exists a $(1,C)$ quasi-isometry between $A_{\phi}$ and the corresponding finite index subgroup $\RAAG$.
\end{prop}

\begin{proof}
    Consider the inclusion map $f \colon \RAAG \rightarrow A_{\phi}$. Then we note that for all $u,v \in \RAAG$, $0 \leq k \leq m-1$, 
    \[ d_{\RAAG}(u,v) = d_{A_{\phi}}(ut^{k}, vt^{k}).
    \]
    Therefore for $0 \leq k,l \leq m-1$,
    \[ d_{A_{\phi}}(ut^{k}, vt^{l}) \geq d_{A_{\phi}}(ut^{k}, vt^{k}) = d_{\RAAG}(u,v).
    \]
    For the upper bound, recall $t^{-1}s_{i}t = \phi(s_{i})$ for all $s_{i} \in X$, and so
    \[ d_{A_{\phi}}(ut^{k}, vt^{l}) \leq d_{A_{\phi}}(ut^{k}, \phi^{-(k-l)}(u)t^{l}) + d_{A_{\phi}}(\phi^{-(k-l)}(u)t^{l}
    \]
\end{proof}
}
\begin{thm}\label{cor:extension transc contract}
    Let $A_{\phi}$ be a virtual RAAG as in \cref{semidirect1}, such that $\RAAG$ is not a direct product or cyclic, and $\phi \in \mathrm{Aut}(\RAAG)$ is length-preserving. Then the conjugacy growth series of $A_{\phi}$ is transcendental, with respect to the word metric.
\end{thm}

\begin{proof}
    By \cref{prop:RAAG product contracting} and \cref{thm:rank 1 cat0 contract}, $\RAAG$ acts on the Salvetti complex $S$ with a contracting element (when $\RAAG$ is not a direct product), with respect to the word metric. When we consider the finite extension $A_{\phi}$, this action is extended to the same space (see \cref{cor:CP solvable RAAGs len p CAT0}). Moreover, for any geodesic $g \in \mathrm{Cay}\left(A_{\phi}, \widehat{X}\right)$, where $\widehat{X} = X \cup \{t\}^{\pm}$, we can project $g$ onto $S^{1}$, which then gives us a contracting element with respect to the action on $\mathrm{Cay}\left(A_{\phi}, \widehat{X}\right)$. Therefore the conjugacy growth series of $A_{\phi}$ is transcendental, using \cref{thm:contracting conj growth}.
\end{proof}

\comm{
However, it remains unclear whether \cref{thm:contracting conj growth} can be used in this context - with further investigation, it is unclear whether the action which contains a contracting element must be with respect to the Cayley graph of the group, rather than an arbitrary space, in order for the group to have transcendental conjugacy growth.}
\comm{
\begin{rmk}
    \cref{cor:extension transc contract} implies that the language $\mathsf{ConjSL}\left(A_{\phi}, \widehat{X}\right)$ is not unambiguous context-free. This is a weaker version of \myRed{CITE LANGUAGES PAPER}.
\end{rmk}}
Similar results hold for RACGs, using \cref{thm:RACGs CAT}.
\comm{
\begin{prop}
    A RACG $W_{\Gamma}$ is virtually a direct product if and only if 
    \begin{equation}\label{eqn:RACG direct product}
        W_{\Gamma} \cong W_{1} \times \left(W_{2} \times \Z^{|K|}_{2}\right),
    \end{equation}
    for some RACGs $W_{1}, W_{2}$, where $K \subseteq \Gamma$ is a (possibly empty) clique.
\end{prop}}

\begin{prop}\cite[Proposition 2.11]{behrstock_thickness_2017}\cite[Theorem 2.14]{charney_contracting_2015}
    A RACG $W_{\Gamma}$ contains a contracting element if and only if $W_{\Gamma}$ is not virtually a direct product. 
\end{prop}
\begin{thm}\label{thm:RACG conj growth}
    Let $W_{\phi} = W_{\Gamma} \rtimes_{\phi} \Z/m\Z$ be a virtual RACG as in \cref{semidirect1}, such that $W_{\Gamma}$ is not a virtual direct product. Then the conjugacy growth series of $W_{\Gamma}$ is transcendental, with respect to the word metric.
\end{thm}

\subsection{RAAGs with finite Outer automorphism group}
The following result has both positive and negative consequences for solving the conjugacy problem in virtual RAAGs.
\begin{prop}(\cite[Corollary 1.8]{huang_groups_2018})
Let $A_{\Gamma}$ be a RAAG such that $\mathrm{Out}(A_{\Gamma})$ is finite. If $G$ is any finitely generated group quasi-isometric to $A_{\Gamma}$, then $G$ is a $\mathrm{CAT}(0)$ group. 
\end{prop}
This implies that for any virtual RAAG $G$, such that $A_{\Gamma} \leq G$ is of finite index and $\RAAG$ has finite $\mathrm{Out}(A_{\Gamma})$, then $G$ has solvable conjugacy problem. For large graphs, Charney and Farber gave a probabilistic argument that $\mathrm{Out}(\RAAG)$ is almost always finite, as the number of vertices in the defining graph $\Gamma$ tends to infinity. 
\begin{thm}\cite[Theorem 6.1]{charney_random_2012}
Let $\Gamma \in G(n,p)$ be a random graph with probability parameter $p$ independent of $n$ such that
\[ 1 - \frac{1}{\sqrt{2}} < p < 1.
\]
Then $A_{\Gamma}$ has finite $\mathrm{Out}(A_{\Gamma})$ asymptotically almost surely.
\end{thm}
This bound was later improved by Day \cite{day_finiteness_2012}. On the other hand, when the number of vertices in the defining graph is small, it is more likely that $\mathrm{Out}(A_{\Gamma})$ will be infinite. This can be seen by the following criteria, which was formalised by Wade.
\begin{prop}\cite[Proposition 2.4]{wade_note_2023}
Let $\Gamma$ be a finite graph. Then $\mathrm{Out}(A_{\Gamma})$ is finite if and only if for each vertex $u \in V(\Gamma)$:
\begin{enumerate}
    \item[(i)] any two vertices in $\Gamma \setminus \mathrm{St}(u)$ are connected by a path in $\Gamma \setminus \mathrm{St}(u)$, and
    \item[(ii)] if $\mathrm{Lk}(u) \subset \mathrm{St}(v)$ for some vertex $v \in V(\Gamma)$, then $u = v$.
\end{enumerate}
\end{prop}
Using this criteria, one can show the following.
\begin{lemma}
Let $\Gamma$ be a finite graph on $\leq 5$ vertices. Then $\mathrm{Out}(A_{\Gamma})$ is finite if and only if $\Gamma = C_{5}$ is the cycle graph on 5 vertices.
\end{lemma}

\subsection{Virtual direct products}
Solvability of equations in groups leads to a positive solution to the conjugacy problem. In particular, if the equation $X^{-1}uXv^{-1} = 1_{G}$ is decidable for all elements $u,v \in G$ and some variable $X$, then we can determine whether or not $u \sim v$ in $G$ for all elements $u,v \in G$. This allows us to use the following result with respect to RAAGs. 

\begin{thm}\cite[Theorem 3.1]{ciobanu_equations_2020}\label{thm:Ciobanu Eq}
    Let $G$ be a group which contains a subgroup $A \times H_{1} \times \dots \times H_{n}$ of finite index, where $A$ is virtually abelian and each $H_{i}$ are non-elementary hyperbolic. Then the conjugacy problem $\mathrm{CP}(G)$ is solvable.
\end{thm}

\begin{cor}\label{cor:prod conj hyp}
    Let $A_{\Gamma}$ be a RAAG such that $A_{\Gamma} \cong A \times H_{1} \times \dots \times H_{n}$, where $A$ is virtually abelian and each $H_{i}$ is isomorphic to a free group. Then any finite extensions $A_{\phi} = \RAAG \rtimes_{\phi} \langle t \rangle$, as in \cref{semidirect1}, have solvable conjugacy problem.
\end{cor}

\begin{exmp}\label{exmp:eqns direct product}
    Let $\Gamma = K_{m,n}$ be a complete bipartite graph. Then $A_{\Gamma} \cong F_{m} \times F_{n}$, and so any finite extension $A_{\phi}$, as in \cref{semidirect1}, will have solvable conjugacy problem by \cref{cor:prod conj hyp}.
\end{exmp}
\begin{exmp}
    Let $\Gamma$ be a defining graph for a RACG such that $W_{\Gamma} = W_{\Gamma_{1}} \times \dots \times W_{\Gamma_{n}}$, where each $\Gamma_{i}$ ($1 \leq i \leq n)$ does not contain any induced squares. By \cref{RACG:hyperbolic}, each $W_{\Gamma_{i}}$ is hyperbolic, and so any finite extension $W_{\phi}$, as in \cref{semidirect1}, will have solvable conjugacy problem by \cref{cor:prod conj hyp}.
\end{exmp}
The following result is a consequence of \cite[Proposition 2.4]{senden_twisted_2021}, which shows that twisted conjugacy can behave well under direct products. 
\begin{prop}
    Let $G = \times^{n}_{i} G_{i}$ be a direct product of groups $G_{i}$. Define the diagonal subgroup $\mathrm{Diag}(G) \leq \mathrm{Aut}(G)$ to be all maps of the form
\[ (g_{1}, \dots, g_{n}) \mapsto (\varphi_{1}(g_{1}), \dots, \varphi_{n}(g_{n})),
\]
where each $\varphi_{i} \in \mathrm{Aut}(G_{i})$ for all $1 \leq i \leq n$. Then
\[ (g_{1}, \dots, g_{n}) \sim_{\varphi} (h_{1}, \dots, h_{n}) \Leftrightarrow g_{i} \sim_{\varphi_{i}} h_{i} \; \text{for all} \; 1 \leq i \leq n.
\]
\end{prop}
\begin{cor}
    Let $G = \times^{n}_{i} G_{i}$ be a direct product of groups $G_{i}$. If $\varphi \in \mathrm{Diag}(G)$, then $\mathrm{TCP}_{\varphi}(G)$ is solvable if and only if $\mathrm{TCP}_{\varphi_{i}}(G_{i})$ is solvable for all $1 \leq i \leq n$.
\end{cor}

\begin{exmp}
    Let $A_{\Gamma} = F_{m} \times F_{n}$, i.e. $\Gamma = K_{m,n}$, and let $\varphi \in \mathrm{Diag}(A_{\Gamma})$. Then the $\mathrm{TCP}_{\varphi}(A_{\Gamma})$ is solvable, since the twisted conjugacy problem is solvable in free groups (see \cref{thm:TCP free}). This provides an alternative partial solution to \cref{exmp:eqns direct product}. 
\end{exmp}

\subsection{Extensions which are graph products}
By \cite[Theorem 3.24]{green_graph_1990}, a graph product has solvable conjugacy problem if and only if each vertex group has solvable conjugacy problem. In several cases, the finite extension $A_{\phi}$ as in \cref{semidirect1} is a graph product of cyclic groups. We provide examples of extensions of RAAGs, with respect to both length-preserving and non-length preserving automorphisms, which have this property. 
\comm{
\begin{rmk}
    In our presentation from \cref{semidirect1}, we will assume that the order of the generator $t$ of the cyclic group is equal to the order of the automorphism $\phi \in \mathrm{Aut}(\RAAG)$.
\end{rmk}}
\begin{prop}\label{prop:examples extensions graph products}
    There exists RAAGs $\RAAG$ and automorphisms $\phi \in \mathrm{Aut}(\RAAG)$, where $\phi$ is an inversion, graph automorphism or order two partial conjugation, such that the extension $A_{\phi}$ as in \cref{semidirect1} is a graph product of cyclic groups.
\end{prop}

\begin{proof}
    We will consider three examples of RAAGs and automorphisms where the extension $A_{\phi}$ is a graph product of cyclic groups.

    First, let $\RAAG$ be a RAAG, and consider the following composition of inversions. Let $D \subseteq V(\Gamma)$ such that the induced subgraph $\Gamma_{D}$ generated by the vertices in $D$ is completely disconnected. We can define the following automorphism:
\begin{align*}
    \phi_{1}: X &\rightarrow X \\
    a &\mapsto a^{-1} \quad \text{for all} \; a \in D, \\
    b &\mapsto b \quad \quad \text{for all} \; b \in X \setminus D.
\end{align*}
Our presentation for $A_{\phi_{1}} = A_{\Gamma} \rtimes_{\phi_{1}} \mathbb{Z}/2\Z$ becomes
\[ P = \langle V(\Gamma), t \mid [s_{i},s_{j}] = 1 \; \forall \; \{s_{i},s_{j}\} \in E(\Gamma), \; t^{2} = 1, R_{1}, R_{2}  \rangle,
\]
where $R_{1}$ denotes the set of relations of the form $[t,s_{i}] = 1$ for all $s_{i} \in X \setminus D$, and $R_{2}$ denotes the set of relations of the form $(ts_{i})^{2} = 1$ for all $s_{i} \in D$. Let $E_{R}$ denote the set of relations of the form $[s_{i}, s_{j}] = 1$ for all edges $\{s_{i}, s_{j}\} \in E(\Gamma)$. 

Let $u_{i} = ts_{i}$, for all $s_{i} \in D$. We can add new generators $u_{i}$, and remove the generators $s_{i}$ using this relation. Our presentation becomes
\[ P = \langle V(\Gamma) \setminus D, \{u_{i}\}, t \mid \overline{E}_{R}, R_{1}, t^{2} = 1, u^{2}_{i} = 1 \rangle,
\]
where $\overline{E}_{R}$ is the set of relations from $E_{R}$, where any $s_{i} \in D$ have been replaced by $tu_{i}$. This gives us relations of the form $[tu_{i}, s_{j}] = 1,$ where $s_{j} \not \in D$ and $[s_{i}, s_{j}] = 1$ in our original presentation. Note $[t, s_{j}] = 1$ for all $s_{j} \not \in D$, and so we can write 
\[ 1 = [tu_{i}, s_{j}] = u_{i}ts^{-1}_{j}tu_{i}s_{j} = u_{i}s^{-1}_{j}u_{i}s_{j}.
\]
Therefore we can rewrite all relations of the form $[tu_{i}, s_{j}] = 1$ as $[u_{i}, s_{j}] = 1$. This gives us a presentation of a graph product of cyclic groups. 

For the second example, let $\Gamma$ be a path of odd length, i.e. $\Gamma$ is of the form
    \begin{center}
        \begin{tikzpicture}
        \filldraw[black] (0,0) circle (1.5pt);
        \filldraw[black] (1,0) circle (1.5pt);
        \filldraw[black] (2,0) circle (1.5pt);
        \filldraw[black] (-1,0) circle (1.5pt);
        \filldraw[black] (-2,0) circle (1.5pt);
        \node at (3,0) {\ldots};
        \node at (-3,0) {\ldots};
        \filldraw[black] (4,0) circle (1.5pt);
        \filldraw[black] (-4,0) circle (1.5pt);
        \draw (0,0) -- (1,0);
        \draw (0,0) -- (-1,0);
        \draw (2,0) -- (1,0);
        \draw (-2,0) -- (-1,0);
        \draw (-4, 0) -- (-3.5, 0);
        draw (4, 0) -- (3.5, 0);
        \node at (0, 0.3) {$x$};
        \node at (1, 0.3) {$x_{1R}$};
        \node at (-1, 0.3) {$x_{1L}$};
        \node at (2, 0.3) {$x_{2R}$};
        \node at (-2, 0.3) {$x_{2L}$};
        \node at (4, 0.3) {$x_{kR}$};
        \node at (-4, 0.3) {$x_{kL}$};
    \end{tikzpicture}
    \end{center}
    Let $\phi_{2} \colon x_{iL} \leftrightarrow x_{iR}$ be a reflection (for all $1 \leq i \leq k)$. Our presentation for $A_{\phi_{2}}$ is 
    \begin{align*}
        P = \langle &x, \{x_{iL}, x_{iR} \mid i = 1, \dots, k\}, t \mid [x, x_{1L}] = 1, [x, x_{1R}] = 1, t^{2} = 1, [x,t] = 1, \\ 
        & [x_{iL}, x_{(i+1)L}] = 1, [x_{iR}, x_{(i+1)R}] \; (i = 1, \dots, k-1), \; tx_{iL}t = x_{iR} \; (i = 1, \dots k) \; \rangle.
    \end{align*}
    We can replace all generators $x_{iR}$ with $tx_{iL}t$, and remove the final relation, to get 
    \begin{align*}
        P = \langle &x, \{x_{iL} \mid i=1,\dots, k\}, t \mid [x, x_{1L}] = 1, [x, tx_{1L}t] = 1, t^{2} = 1, [x,t] = 1, \\
        &[x_{iL}, x_{(i+1)L}] = 1, [tx_{iL}t, tx_{(i+1)L}t] = 1 \; (i = 1, \dots, k-1) \;\rangle.
    \end{align*}
    All relations of the form $[tx_{iL}t, tx_{(i+1)L}t] = 1$ can be removed, using the remaining relations. Similarly the relation $[x, tx_{1L}t] = 1$ can be removed, and so we are left with
    \begin{align*}
        P = \langle &x, \{x_{iL} \mid i=1,\dots, k\}, t \mid [x, x_{1L}] = 1,  t^{2} = 1, [x,t] = 1, \\
        &[x_{iL}, x_{(i+1)L}] = 1 \; (i = 1, \dots, k-1) \; \rangle.
    \end{align*}
    Again we have a graph product of cyclic groups.

    Finally recall \cite[Theorem 4.12]{crowe_conjugacy_2023}. We have an example of a virtual RAAG $A_{\phi}$ where $\phi \in \mathrm{Aut}(\RAAG)$ is a finite order non-length preserving automorphisms, and $A_{\phi}$ is isomorphic to a graph product of cyclic groups. 
\end{proof}
\mixCPRAAGs
\comm{
\begin{exmp}\label{thm:tietze inversions graph product}
    We can construct a finite extension of any RAAG $\RAAG$ by a composition of inversions, such that $A_{\phi}$ as in \cref{semidirect1} is a graph product with solvable conjugacy problem.

    Let $D \subseteq X$ such that the induced subgraph $\Gamma_{D}$ generated by the vertices in $D$ is completely disconnected. We can define the following automorphism:
\begin{align*}
    \phi: X &\rightarrow X \\
    a &\mapsto a^{-1} \quad \text{for all} \; a \in D, \\
    b &\mapsto b \quad \quad \text{for all} \; b \in X \setminus D.
\end{align*}
Our presentation for $A_{\phi} = A_{\Gamma} \rtimes_{\phi} \mathbb{Z}/2\Z$ becomes
\[ P = \langle X, t \mid [s_{i},s_{j}] = 1 \; \forall \; \{s_{i},s_{j}\} \in E(\Gamma), \; t^{2} = 1, R_{1}, R_{2}  \rangle,
\]
where $R_{1}$ denotes the set of relations of the form $[t,s_{i}] = 1$ for all $s_{i} \in X \setminus D$, and $R_{2}$ denotes the set of relations of the form $(ts_{i})^{2} = 1$ for all $s_{i} \in D$. Let $E_{R}$ denote the set of relations of the form $[s_{i}, s_{j}] = 1$ for all edges $\{s_{i}, s_{j}\} \in E(\Gamma)$. 

Let $u_{i} = ts_{i}$, for all $s_{i} \in D$. We can insert new generators $u_{i}$, and remove the generators $s_{i}$ using this relation. Our presentation becomes
\[ P = \langle X \setminus D, \{u_{i}\}, t \mid \overline{E}_{R}, R_{1}, t^{2} = 1, u^{2}_{i} = 1 \rangle,
\]
where $\overline{E}_{R}$ is the set of relations from $E_{R}$, where any $s_{i} \in D$ have been replaced by $tu_{i}$. This gives us relations of the form $[tu_{i}, s_{j}] = 1,$ where $s_{j} \not \in D$ and $[s_{i}, s_{j}] = 1$ in our original presentation. Note $[t, s_{j}] = 1$ for all $s_{j} \not \in D$, and so we can write 
\[ 1 = [tu_{i}, s_{j}] = u_{i}ts^{-1}_{j}tu_{i}s_{j} = u_{i}s^{-1}_{j}u_{i}s_{j}.
\]
Therefore we can rewrite all relations of the form $[tu_{i}, s_{j}] = 1$ as $[u_{i}, s_{j}] = 1$. This gives us a presentation of a graph product of cyclic groups. 
\end{exmp}}

\comm{
\begin{prop}\label{thm:tietze inversions graph product}
There exist finite extensions of RAAGs $A_{\Gamma}$ by a composition of inversions, such that $A_{\phi}$ as in \cref{semidirect1} is a graph product with solvable conjugacy problem.
\end{prop}

\begin{proof}[Proof of \cref{thm:tietze inversions graph product}]
We will rewrite our group presentation of $A_{\phi}$ from \cref{semidirect1}, using Tietze transformations, to obtain a presentation of a graph product of cyclic groups. Let $D \subseteq X$, such that the induced subgraph $\Gamma_{D}$ generated by the vertices in $D$ is completely disconnected. We can define the following automorphism:
\begin{align*}
    \phi: X &\rightarrow X \\
    a &\mapsto a^{-1} \quad \text{for all} \; a \in D, \\
    b &\mapsto b \quad \quad \text{for all} \; b \in X \setminus D.
\end{align*}
Our presentation for $A_{\phi} = A_{\Gamma} \rtimes_{\phi} \mathbb{Z}/2\Z$ becomes
\[ P = \langle X, t \mid [s_{i},s_{j}] = 1 \; \forall \; \{s_{i},s_{j}\} \in E(\Gamma), \; t^{2} = 1, R_{1}, R_{2}  \rangle,
\]
where $R_{1}$ denotes the set of relations of the form $[t,s_{i}] = 1$ for all $s_{i} \in X \setminus D$, and $R_{2}$ denotes the set of relations of the form $(ts_{i})^{2} = 1$ for all $s_{i} \in D$. Let $E_{R}$ denote the set of relations of the form $[s_{i}, s_{j}] = 1$ for all edges $\{s_{i}, s_{j}\} \in E(\Gamma)$. 

Let $u_{i} = ts_{i}$, for all $s_{i} \in D$. We can insert new generators $u_{i}$, and remove the generators $s_{i}$ using this relation. Our presentation becomes
\[ P = \langle X \setminus D, \{u_{i}\}, t \mid \overline{E}_{R}, R_{1}, t^{2} = 1, u^{2}_{i} = 1 \rangle,
\]
where $\overline{E}_{R}$ is the set of relations from $E_{R}$, where any $s_{i} \in D$ have been replaced by $tu_{i}$. This gives us relations of the form $[tu_{i}, s_{j}] = 1,$ where $s_{j} \not \in D$ and $[s_{i}, s_{j}] = 1$ in our original presentation. Note $[t, s_{j}] = 1$ for all $s_{j} \not \in D$, and so we can write 
\[ 1 = [tu_{i}, s_{j}] = u_{i}ts^{-1}_{j}tu_{i}s_{j} = u_{i}s^{-1}_{j}u_{i}s_{j}.
\]
Therefore we can rewrite all relations of the form $[tu_{i}, s_{j}] = 1$ as $[u_{i}, s_{j}] = 1$. This gives us a presentation of a graph product of cyclic groups. 
\end{proof}}
\comm{
There exist several examples of extensions of RAAGs with respect to graph automorphisms, which can be presented as graph products of cyclic groups. We provide a general example here, where the defining graph of our RAAG is a path of odd length. 

\begin{prop}
    Let $\Gamma$  Then the extension $A_{\phi}$ as in \cref{semidirect1} is a graph product with solvable conjugacy problem. 
\end{prop}

\begin{proof}
    Our presentation for $A_{\phi}$ is 
    \begin{align*}
        P = \langle &x, \{x_{iL}, x_{iR}\}, t \mid [x, x_{1L}] = 1, [x, x_{1R}] = 1, \\ 
        &[x_{iL}, x_{(i+1)L}] = 1, [x_{iR}, x_{(i+1)R}], t^{2} = 1, [x,t] = 1, tx_{iL}t = x_{iR} \rangle.
    \end{align*}
    We can replace all generators $x_{iR}$ with $tx_{iL}t$, and remove the final relation, to get 
    \begin{align*}
        P = \langle &x, \{x_{iL}\}, t \mid [x, x_{1L}] = 1, [x, tx_{1L}t] = 1, \\
        &[x_{iL}, x_{(i+1)L}] = 1, [tx_{iL}t, tx_{(i+1)L}t] = 1, t^{2} = 1, [x,t] = 1 \rangle.
    \end{align*}
    All relations of the form $[tx_{iL}t, tx_{(i+1)L}t] = 1$ can be removed, using the remaining relations. Similarly the relation $[x, tx_{1L}t] = 1$ can be removed, and so we are left with
    \[ P = \langle x, \{x_{iL}\}, t \mid [x, x_{1L}] = 1, [x_{iL}, x_{(i+1)L}] = 1, t^{2} = 1, [x,t] = 1 \rangle.
    \]
    This is precisely a graph product of cyclic groups, which has solvable conjugacy problem.
\end{proof}}
\comm{
We finish this section with an example of a finite extension by a non-length preserving automorphism. 
\nonlengthCP
\begin{proof}
    By \cref{thm:non len geos}, $A_{\phi}$ is isomorphic to a graph product of cyclic groups, which has solvable conjugacy problem. 
\end{proof}}

\section{Pilings}\label{sec:pilings}
The aim of this section is to prove the following.
\TCPRAAG
Recall decidability has already been shown by \cref{cor:CP solvable RAAGs len p CAT0} and \cref{cor:CP extensions iff TCP in base}. 

\begin{rmk}\label{rmk:RAM machines}
    Complexity of algorithms will be determined using a random access memory (RAM) machine, where basic arithmetical operations on integers are assumed to take constant time (see \cite[Section 1.2]{aho_design_1974} for further information). With this machinery, computations such as determining if two words are cyclic permutations of each other run in linear time, using standard pattern-matching algorithms \cite[Section 9]{aho_design_1974}. Note the same computation takes time $\mathcal{O}(\ell \mathrm{log}(\ell))$ on a Turing machine.
\end{rmk}

Pilings are a geometric tool which can be used to represent group elements uniquely in a RAAG. The inspiration for this object comes from a paper by Viennot \cite{labelle_heaps_1986}, which describes a geometric construction known as `heaps of pieces' for partially commutative monoids. This was adapted in \cite{crisp_conjugacy_2009} for RAAGs, and used to find a linear time solution to the conjugacy problem in RAAGs. We will summarise this algorithm and provide necessary definitions and results about pilings, before adapting this algorithm to solve \cref{thm:len p RAAGs pilings}.

We remind the reader that our convention is the opposite of \cite{crisp_conjugacy_2009}, in that an edge exists between two vertices $u, v \in V(\Gamma)$ in the defining graph $\Gamma$ if and only $u$ and $v$ commute in $A_{\Gamma}$. To recap notation, we let $V(\Gamma) = \{s_{1}, \dots, s_{r}\}$ denote the standard generating set for a RAAG $\RAAG$, and we let $X = V(\Gamma)^{\pm}$. 

\begin{defn}\label{defn:abstract piling}
An \emph{abstract piling} is an $r$-tuple of words, one for each vertex $s_{i} \in V(\Gamma)$ ($1 \leq i \leq r$), over the alphabet $\Sigma = \{+, -, 0\}$. We define the word associated with each vertex $s_{i} \in V(\Gamma)$ to be the \emph{$\sigma_{i}$-stack}. We define a function $\sigma \colon \{-1,0,1\} \rightarrow \Sigma$ which maps $1 \mapsto +, -1 \mapsto -,$ and $0 \mapsto 0$. Let $w = s^{\varepsilon_{i_{1}}}_{i_{1}}\dots s^{\varepsilon_{i_{n}}}_{i_{n}}  \in X^{\ast}$ be a word which represents a group element of $\RAAG$, where $\varepsilon_{i_{j}} = \pm 1$ and $s_{i_{j}} \in V(\Gamma)$, for all $1 \leq j \leq n$. We define a function $\pi^{\ast}$ which maps $w$ to an abstract piling as follows. 

First, let each $\sigma_{i}$-stack equal the empty word (for all $1 \leq i \leq r$). Reading $w$ from left to right, suppose the letter $s^{\varepsilon_{i_{j}}}_{i_{j}}$ is read, for some $1 \leq j \leq n$. Let $\alpha_{i_{j}}$ be the last letter of the $\sigma_{i_{j}}$-stack. We have two cases to consider.
\begin{enumerate}
    \item \textbf{Empty stack}: If $\sigma_{i_{j}}$ equals the empty word, then we rewrite the stack $\sigma_{i_{j}} \mapsto \sigma(\varepsilon_{i_{j}})$.
    \item \textbf{No cancellation}: If $\alpha_{i_{j}} \neq \sigma(-\varepsilon_{i_{j}})$, then we rewrite the stack $\sigma_{i_{j}} \mapsto \sigma_{i_{j}}\cdot \sigma(\varepsilon_{i_{j}})$, and rewrite stacks $\sigma_{k} \mapsto \sigma_{k}\cdot 0,$ for all $1 \leq k \neq i_{j} \leq r$ such that $[s_{i_{j}}, s_{k}] \neq 1$ in $A_{\Gamma}$. 
    \item \textbf{Cancellation}: If $\alpha_{i_{j}} = \sigma(-\varepsilon_{i_{j}})$, i.e. $\sigma_{i_{j}} = v\cdot \sigma(-\varepsilon_{i_{j}})$ for some word $v \in \Sigma^{\ast}$. This implies that $\sigma_{k} = v_{k}\cdot 0$ for some word $v_{k} \in \Sigma^{\ast}$, for all $1 \leq k \neq i_{j} \leq r$ such that $[s_{i_{j}}, s_{k}] \neq 1$ in $\RAAG$. We rewrite the stack $\sigma_{i_{j}} \mapsto v,$ and all stacks $\sigma_{k} \mapsto v_{k}$ for all $1 \leq k \neq i_{j} \leq r$ such that $[s_{i_{j}}, s_{k}] \neq 1$ in $\RAAG$. 
\end{enumerate}
Once all letters from $w$ have been read, we have the associated \emph{piling} of $w$, which we denote by $\pi^{\ast}(w) = (\sigma_{1}, \dots, \sigma_{r}) \in P$, where $P$ denotes the set of all possible pilings. 
\end{defn}
It is easier to think of pilings using the following geometric definition.
\begin{defn}(Geometric definition of pilings)\\
Start with $r$ vertical stacks, labelled $\sigma_{i}$ for each of the vertices $s_{i} \in V(\Gamma)$ ($1 \leq i \leq r$). Let $s^{\varepsilon_{i}}_{i} \in X$ ($\varepsilon_{i} = \pm 1$) be a letter of a word $w \in X^{\ast}$, for some $1 \leq i \leq r$. Then $s^{\varepsilon_{i}}_{i}$ is associated to a collection of beads, called a \emph{tile}, corresponding to one bead labelled $\sigma(\varepsilon_{i})$ on the $\sigma_{i}$-stack, and one bead labelled 0 on each $\sigma_{j}$-stack such that $1 \leq j \neq i \leq r$ and $[s_{i}, s_{j}] \neq 1$ in $A_{\Gamma}$. Each of these 0-beads is connected to the $\sigma(\varepsilon_{i})$-bead by a \emph{thread}. Note $0$-beads in a tile can commute with each other on a stack, but cannot commute with $+$ or $-$ beads.

When constructing $\pi^{\ast}(w)$, we add $+,-,0$ beads on stacks as given in \cref{defn:abstract piling}, based on the exponents of letters in $w$. If cancellation occurs, that is, $s^{\varepsilon_{i}}_{i}s^{-\varepsilon_{i}}_{i}$ occurs in $w$ up to commutation relations (for some $1 \leq i \leq r$), then we have consecutive beads $+-$ or $-+$ on the $\sigma_{i}$-stack, which cancel. When this occurs, we remove the pair of tiles associated to these beads, that is, we remove both the $+-$ or $-+$ beads, and their associated $0$ beads.
\end{defn}
\comm{
\begin{rmk}
    When considering stacks in a piling, we read beads from the bottom to the top of the piling.
\end{rmk}
See \cite[Example 2.4]{crisp_conjugacy_2009} for some examples of pilings.}
Pilings are an excellent geometric tool for RAAGs, in that any two words representing the same group element have the same piling associated to them. Note equivalent pilings refers to the beads on each stack, rather than how threads connect beads together. More formally, we say two pilings $\pi^{\ast}(u) = (\sigma_{1}, \dots, \sigma_{r}),$ $\pi^{\ast}(v) = (\tau_{1}, \dots, \tau_{r}) \in P$ are equal if and only if $\sigma_{i} = \tau_{i}$ for all $1 \leq i \leq r$.

\begin{prop}\cite{crisp_conjugacy_2009}\label{prop:shuffles pilings}
The map $\pi^{\ast}$ induces a well-defined map $\pi\colon \RAAG \rightarrow P$. In particular, for two words $u,v \in X^{\ast}$, then $u =_{\RAAG} v$ if and only if $\pi^{\ast}(u) = \pi^{\ast}(v)$. 
\end{prop}

\begin{proof}
    Recall any two geodesics representing the same element of $\RAAG$ are related by a finite number of commutation relations. It therefore remains to prove that removal of a free cancellation or applying a commutation relation does not change the image of a word under $\pi^{\ast}$. The first case is equivalent to showing that $\pi(s_{i}^{\varepsilon}s_{i}^{-\varepsilon}) = \pi(\epsilon)$ for all $1 \leq i \leq r$, where $s_{i} \in V(\Gamma),$ $\varepsilon = \pm 1$, and $\epsilon$ denotes the empty word. This is immediate by definition, since the piling $\pi(s_{i}^{\varepsilon}s_{i}^{-\varepsilon})$ consists of consecutive $+-$ or $-+$ beads on the $\sigma_{i}$-stack, which must cancel. 

    We now prove that for vertices $s_{i}, s_{j} \in V(\Gamma)$ such that $[s_{i}, s_{j}] = 1$ $(i \neq j)$ , we have $\pi\left(s_{i}^{\varepsilon_{i}}s_{j}^{\varepsilon_{j}}\right) = \pi\left(s_{j}^{\varepsilon_{j}}s_{i}^{\varepsilon_{i}}\right)$, where $\varepsilon_{i}, \varepsilon_{j} = \pm 1$. First suppose $(V(\Gamma) \setminus \mathrm{Lk}(s_{i})) \cap (V(\Gamma) \setminus \mathrm{Lk}(s_{j})) = \varnothing$. The piling $\pi(s^{\varepsilon_{i}}_{i})$ consists of a $\sigma(\varepsilon_{i})$-bead on the $\sigma_{i}$-stack, and 0-beads on every stack which does not lie in $\mathrm{St}(s_{i})$. Similarly the piling $\pi(s^{\varepsilon_{j}}_{j})$ consists of a $\sigma(\varepsilon_{j}$)-bead on the $\sigma_{j}$-stack, and 0-beads on every stack which does not lie in $\mathrm{St}(s_{j})$. Since $s_{i}, s_{j}, V(\Gamma) \setminus \mathrm{St}(s_{i})$ and $V(\Gamma) \setminus \mathrm{St}(s_{j})$ are all disjoint, the order in which we concatenate $\pi(s^{\varepsilon_{i}}_{i})$ and $\pi\left(s^{\varepsilon_{j}}_{j}\right)$ does not change the output piling. Otherwise, suppose there exists a $\sigma_{k}$-stack,
    such that $s_{k} \in V(\Gamma) \setminus (\mathrm{Lk}(s_{i}) \cup \mathrm
    {Lk}(s_{j}))$. Again the order of concatenation of pilings does not matter, since 0-beads can commute with each other in a piling (see \cref{fig:shuffles equiv pilings}). 
    \begin{figure}[h]
        \centering
        \begin{tikzpicture}
            \filldraw[black] (0,0) rectangle (3, 0.2);
            \draw[black] (0.5, 0.2) -- (0.5, 2);
            \node at (0.5, 2.2) {$s_{i}$};
            \node at (1.5, 2.2) {$s_{k}$};
            \node at (2.5, 2.2) {$s_{j}$};
            \draw[black] (1.5,0.2) -- (1.5,2) ;
            \draw[black] (2.5,0.2) -- (2.5,2) ;
            \filldraw[black] (5,0) rectangle (8, 0.2);
            \node at (5.5, 2.2) {$s_{i}$};
            \node at (6.5, 2.2) {$s_{k}$};
            \node at (7.5, 2.2) {$s_{j}$};
            \draw[black] (5.5,0.2) -- (5.5,2) ;
            \draw[black] (6.5,0.2) -- (6.5,2) ;
            \draw[black] (7.5,0.2) -- (7.5,2) ;
            \node at (4,1) {\textbf{=}};
            \filldraw[color = black, fill=red] (0.5, 0.5) circle (6pt);
            \draw[black] (0.3, 0.5) -- (0.7, 0.5);
            \draw[black] (0.5, 0.2) -- (0.5, 0.8);
            \filldraw[color = black, fill=white] (1.5, 0.5) circle (6pt);
            \filldraw[color = black, fill=red] (2.5, 0.5) circle (6pt);
            \filldraw[color = black, fill=white] (1.5, 1) circle (6pt);
            \draw[black] (2.3, 0.5) -- (2.7, 0.5);
            \draw[black] (2.5, 0.2) -- (2.5, 0.8);
            \draw[blue] (0.7, 0.5) -- (1.3, 0.5);
            \draw[blue] (1.7, 0.9) -- (2.3, 0.6);
            \filldraw[color = black, fill=red] (5.5, 0.5) circle (6pt);
            \draw[black] (5.3, 0.5) -- (5.7, 0.5);
            \draw[black] (5.5, 0.2) -- (5.5, 0.8);
            \filldraw[color = black, fill=white] (6.5, 0.5) circle (6pt);
            \filldraw[color = black, fill=red] (7.5, 0.5) circle (6pt);
            \filldraw[color = black, fill=white] (6.5, 1) circle (6pt);
            \draw[black] (7.3, 0.5) -- (7.7, 0.5);
            \draw[black] (7.5, 0.2) -- (7.5, 0.8);
            \draw[blue] (5.7, 0.6) -- (6.3, 0.9);
            \draw[blue] (6.7, 0.5) -- (7.3, 0.5);
        \end{tikzpicture}
        \caption{Well-defined pilings under commutation}
        \label{fig:shuffles equiv pilings}
    \end{figure}
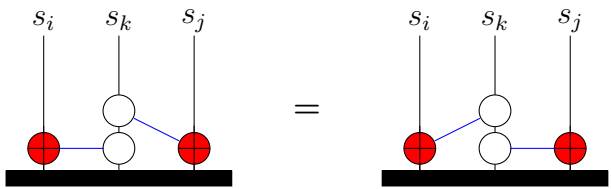
\end{proof}
We now summarise a normal form for pilings given in \cite[Definition 2.5]{crisp_conjugacy_2009}.
\begin{defn}
    Given an order $\leq$ on $X$, let $\leq^{-1}$ denote the \emph{inverse order}. In particular, $x \leq y$ if and only if $y \leq^{-1} x$. Let $\leq^{-1}_{\mathrm{SL}}$ denote the induced shortlex ordering of $X^{\ast}$ with respect to $\leq^{-1}$. We say a word $w \in X^{\ast}$ is \emph{normal} if for all $v \in X^{\ast}$ such that $w =_{\RAAG} v$, then $w \leq^{-1}_{\mathrm{SL}} v$. Any piling $p \in P$ has a unique normal geodesic which represents the group element defined by $p$. 
\end{defn}
\subsection{Conjugacy problem in RAAGs}\label{subsec:CP in RAAGs}
We provide a summary of the algorithm given in \cite{crisp_conjugacy_2009}, where the authors use pilings to find a linear time solution for the conjugacy problem in RAAGs. The first step is to establish a method of cyclically reducing pilings (recall \cref{defn:CR RAAGS}). 

\begin{defn}(Top and bottom tiles)\\
If an $\sigma_{i}$-stack starts with a bead $b \in \{+, -\}$, we define the \emph{bottom $s_{i}$-tile} to be the sub-piling formed by the first bead of the $\sigma_{i}$-stack, and the first beads of the $\sigma_{j}$-stacks such that $[s_{i}, s_{j}] \neq 1$ (which are all $0$-beads by definition). A \emph{top $s_{i}$-tile} is defined analogously for the end of a $\sigma_{i}$-stack.
\end{defn}

\begin{rmk}
    Our definition of top and bottom tiles coincides with the following definitions from \cite{Ferov2016}. For any $g \in \RAAG$, we define $\mathrm{FL}(g) \subseteq V(\Gamma)$ to be the set of all $v \in V(\Gamma)$ such that $g$ can be represented by a geodesic beginning (resp. ending) with $v^{\pm 1}$. Similarly we define $\mathrm{LL}(g) \subseteq V(\Gamma)$ to be the set of all $v \in V(\Gamma)$ such that $g$ can be represented by a geodesic ending with $v^{\pm 1}$. These sets $\mathrm{FL}(g)$ and $\mathrm{LL}(g)$ are precisely the sets of bottom and top tiles respectively of a piling $p \in P$ which represents $g$. 
\end{rmk}

\begin{defn}\label{defn:CR piling}\cite[Definition 2.12]{crisp_conjugacy_2009}
    Let $p \in P$. If an $\sigma_{i}$-stack starts with a $+$ bead and ends with a $-$ bead (or vice versa), then we define a \emph{cyclic reduction} of $p$ to be the removal of the top and bottom $s_{i}$-tiles. A piling $p \in P$ is \emph{cyclically reduced} if no cyclic reduction is possible.
\end{defn} 
See \cite[Figure 3]{crisp_conjugacy_2009} for an example of a cyclic reduction of a piling. The following is then immediate from these definitions.
\begin{cor}
    A reduced word $w \in X^{\ast}$ is cyclically reduced if and only if $\pi^{\ast}(w)$ is cyclically reduced.
\end{cor}
The next step of the algorithm involves considering the defining graph. Again we remind the reader that our convention for the defining graph is the opposite to that of \cite{crisp_conjugacy_2009}.

\begin{defn}\cite[Definition 2.13]{crisp_conjugacy_2009}
    Let $w \in X^{\ast}$ be geodesic and let $\pi^{\ast}(w) = p \in P$. Let $\Gamma^{c}$ denote the complement of the defining graph $\Gamma$. Consider the full subgraph $\Delta(p)$ of $\Gamma^{c}$, defined as the induced subgraph of $\Gamma^{c}$ whose vertices corresponds to $\sigma_{i}$-stacks in $p$ which contain at least one $+$ or $-$ bead, i.e. the subgraph induced by the support of $w$. We say $w$ and $p$ are \emph{non-split} if $\Delta(p)$ is connected.
\end{defn}
\cite[Remark 2.14]{crisp_conjugacy_2009} highlights that the conjugacy problem can be reduced to considering cyclically reduced non-split pilings. This is also observed in \cite[Proposition 5.7]{borovik_divisibility_2005}.

\begin{prop}
    Let $w = w_{1}\dots w_{k} \in X^{\ast}$ be a factorisation such that each induced subgraph $\Delta(w_{i})$ is a connected component of $\Gamma^{c}$, that is, each $w_{i}$ is a non-split word. Then $w$ is cyclically reduced if and only if each $w_{i}$ is cyclically reduced. 

    Moreover, if $v = v_{1}\dots v_{l}$ is cyclically reduced and factorised into non-split words, then $w \sim v$ if and only if $k=l$ and, after index re-enumeration, $w_{i} \sim v_{i}$ for all $1 \leq i \leq k$. 
\end{prop}
The next step is to construct a type of piling, known as a pyramidal piling, such that when we take our normal form, then any cyclic permutations of this normal form are also normal. 

\begin{defn}\cite[Definition 2.15]{crisp_conjugacy_2009}
    Let $p \in P$ be a non-empty piling. Let $i \in \{1, \dots, r\}$ denote the largest index such that the $\sigma_{i}$-stack in $p$ contains a $+$ or $-$ bead. We say $p$ is \emph{pyramidal} if the first bead of the $\sigma_{i}$-stack is a $+$ or $-$ bead, and every other stack in $p$ is either empty or starts with a 0 bead. 
\end{defn}
\begin{cor}
    Any pyramidal piling must be non-split.
\end{cor}
\begin{defn}\label{defn:piling cyclic perm}
    Let $p \in P$ be a non-empty piling. A \emph{cyclic permutation} of $p$ is defined as the operation of removing a bottom (resp. top) $s_{i}$-tile from $p$, and adding a top (resp. bottom) $s_{i}$-tile to $p$.
\end{defn}
 
\begin{prop}\cite[Proposition 2.18]{crisp_conjugacy_2009}
    Any non-split cyclically reduced piling can be transformed into a pyramidal piling  via a finite sequence of cyclic permutations.
\end{prop}
The final step in ensuring a linear time algorithm comes from defining a normal form called a \emph{cyclic normal form}.

\begin{defn}\cite[Definition 2.19]{crisp_conjugacy_2009}
    Let $w \in X^{\ast}$ be reduced and cyclically reduced. We say $w$ is a \emph{cyclic normal form} if it is normal and all its cyclic permutations are also normal.
\end{defn}
\begin{prop}\label{prop: 2.20 RAAGs}\cite[Proposition 2.20]{crisp_conjugacy_2009}
    Let $p \in P$ be a non-split cyclically reduced pyramidal piling. Let $w \in X^{\ast}$ be the unique normal reduced word which represents the group element defined by $p$. Then $w$ is a cyclic normal form.
\end{prop}
\begin{prop}\label{prop: 2.21 RAAGs}\cite[Proposition 2.21]{crisp_conjugacy_2009}
    Two cyclic normal forms represent conjugate elements if and only if they are equal up to a cyclic permutation. 
\end{prop}
We now summarise the linear time algorithm to solve the conjugacy problem in RAAGs. 

\textbf{\underline{Algorithm: Conjugacy problem in RAAGs}}

\textbf{Input}:
\begin{enumerate}
    \item RAAG $A_{\Gamma}$.
    \item Words $v,w \in X^{\ast}$ representing group elements in $A_{\Gamma}$.
\end{enumerate}
\textbf{Step 1: Cyclic reduction}
\begin{adjustwidth}{1.5cm}{}
    Produce the piling representation $\pi^{\ast}(v)$ of $v$, and apply cyclic reduction to $\pi^{\ast}(v)$ to produce a cyclically reduced piling $p$. Repeat this step for $w$ to get a cyclically reduced piling $q$.
\end{adjustwidth}
\textbf{Step 2: Factorisation}
\begin{adjustwidth}{1.5cm}{}
    Factorise each of the pilings $p$ and $q$ into non-split factors. If the collection of subgraphs (using \cite{crisp_conjugacy_2009} convention) do not coincide, \textbf{Output} = \texttt{\color{blue}False}.
\end{adjustwidth}
\textbf{Step 3: Compare non-split factors}
\begin{adjustwidth}{1.5cm}{}
    If $p = p^{(1)}\dots p^{(k)}$ and $q = q^{(1)}\dots q^{(k)}$ are the factorisations found in Step 2, then for each $i = 1, \dots, k$, do the following:
    \begin{adjustwidth}{1.5cm}{}
        \begin{enumerate}[label=(\roman*)]
            \item Transform the non-split cyclically reduced pilings $p^{(i)}$ and $q^{(i)}$ into pyramidal pilings $\Tilde{p}^{(i)}$ and $\Tilde{q}^{(i)}$.
            \item Produce the unique words in cyclic normal form representing these pilings, denoted $\sigma^{\ast}\left(\Tilde{p}^{(i)}\right)$ and $\sigma^{\ast}\left(\Tilde{q}^{(i)}\right)$.
            \item Decide whether $\sigma^{\ast}\left(\Tilde{p}^{(i)}\right)$ and $\sigma^{\ast}\left(\Tilde{q}^{(i)}\right)$ are equal up to a cyclic permutation. If not, \textbf{Output} = \texttt{\color{blue}False}.
        \end{enumerate}
    \end{adjustwidth}
    \textbf{Output} = \texttt{\color{blue}True}.
\end{adjustwidth}
\begin{rmk}
    These definitions can be adapted to find a linear time solution to the conjugacy problem in RACGs. To construct pilings, we would not use any $-$ beads since each generator has order two, and cancel any pairs of consecutive $+$ beads. 
\end{rmk}

\subsection{Twisted conjugacy problem algorithms}
The aim of this section is to reprove \cref{cor:CP solvable RAAGs len p CAT0}, by creating an implementable algorithm for the twisted conjugacy problem in RAAGs, with respect to length-preserving automorphisms. First, we define a linear time algorithm based on \cref{subsec:CP in RAAGs}, when our automorphism is a composition of inversions. It remains unclear whether this algorithm can be adapted to include all length-preserving automorphisms. Instead we provide an alternative algorithm, of which the complexity is unknown, which solves the twisted conjugacy problem in RAAGs for all length-preserving automorphisms. 

First, we define twisted versions of cyclic permutations and cyclic reduction.

\begin{defn}\label{defn:cyc perm}
Let $G = \langle X \rangle$, and let $w = x_{1}\dots x_{n} \in X^{\ast}$ be a geodesic, where $x_{i} \in X$ for all $1 \leq i \leq n$. Let $\phi \in \mathrm{Aut}(G)$ be of finite order $m$. We define a \emph{$\phi$-cyclic permutation} of $w$ to be any word of the form
\[ w' = \phi^{k}(x_{i+1})\dots \phi^{k}(x_{n})\phi^{k-1}(x_{1})\dots \phi^{k-1}(x_{i}),
\]
for some $0 \leq k \leq m-1$.  
\end{defn}
\comm{
We note that if $\phi$ is the trivial automorphism, a $\phi$-cyclic permutation is equivalent to a cyclic permutation.}

\begin{defn}\label{def: CR}
Let $\RAAG = \langle V(\Gamma) \rangle$, and let $v \in X^{\ast}$ be a geodesic. We say $v$ is \emph{$\phi$-cyclically reduced} ($\phi$-CR) if there does not exist a sequence of $\phi$-cyclic permutations, commutation relations and free reductions to a geodesic $w \in X^{\ast}$, such that $l(v) > l(w)$. 
\end{defn}
We also recall the following result which is an analogue of the behaviour of conjugate cyclically reduced elements in RAAGs.
\begin{thm}\cite[Corollary 3.14]{crowe_conjugacy_2023}\label{thm:language sequence}
Let $\RAAG = \langle V(\Gamma) \rangle$ be a RAAG, and let $\phi \in \mathrm{Aut}(\RAAG)$ be of finite order. Let $u,v \in X^{\ast}$ be two $\phi$-cyclically reduced words. Then $u$ and $v$ are $\phi$-conjugate if and only if $u$ and $v$ are related by a finite sequence of $\phi$-cyclic permutations, commutation relations and free reductions.    
\end{thm}

\subsubsection{Inversions}
Throughout this section let $\phi \in \mathrm{Aut}(\RAAG)$ be a composition of inversions. First we consider twisted cyclic reduction in pilings. For inversions, we have the following characterisation of $\phi$-CR words.  
\begin{thm}\cite[Corollary 3.18]{crowe_conjugacy_2023}\label{form:inversions}
    Let $\RAAG = \langle V(\Gamma) \rangle$, and let $\phi \in \mathrm{Aut}(\RAAG)$ be a composition of inversions. Then any geodesic $v \in X^{\ast}$ is $\phi$-CR if and only if $v$ cannot be written in the form 
\[ v =_{\RAAG} \phi(u)^{-1}wu,
\]
where $l(w) < l(v)$ and $\phi(u)^{-1}wu \in X^{\ast}$ is geodesic. 
\end{thm}
With this result we can define $\phi$-cyclic reduction of pilings, equivalent to \cref{def: CR}, with respect to inversions. 
\begin{defn}
    Let $p \in P$. We define a $\phi$-CR of $p$ to be the removal of top and bottom tiles as follows. Either:
    \begin{enumerate}
    \item For each $\sigma_{i}$-stack such that $\phi(s_{i}) = s^{-1}_{i}$, and $\sigma_{i}$ starts and ends with a $+$ bead or $-$ bead, then remove the bottom and top tiles of $\sigma_{i}$, or
    \item For all remaining stacks, apply cyclic reduction as given in \cref{defn:CR piling}.
\end{enumerate}
    A piling $p \in P$ is $\phi$-CR if no $\phi$-cyclic reduction is possible. In particular, we can $\phi$-CR pilings in linear time by \cref{form:inversions}.
\end{defn}
\begin{exmp}\label{exmp: inversions}
    Consider the RAAG defined in \cite[Example 2.4]{crisp_conjugacy_2009} with presentation
    \[A_{\Gamma} = \langle a_{1}, a_{2}, a_{3}, a_{4} \mid [a_{1}, a_{4}] = 1, [a_{2}, a_{3}] = 1, [a_{2}, a_{4}] = 1 \rangle.\]
    Let $\phi \colon a_{2} \mapsto a^{-1}_{2}, a_{4} \mapsto a^{-1}_{4}$ and fix all remaining generators. \cref{fig:twisted CR inversions} gives an example of a $\phi$-CR of a piling. In particular, we $\phi$-CR the element
    \[ u =_{\RAAG} a_{2}a^{-1}_{4}a_{3}a_{1}a_{2}a^{-1}_{1}a_{2}a_{2} = \phi(a_{2})^{-1}a^{-1}_{4}a_{3}a_{1}a_{2}a^{-1}_{1}a_{2}a_{2} \xrightarrow{\phi-\text{CR}} a^{-1}_{4}a_{3}a_{1}a_{2}a^{-1}_{1}a_{2}.
    \]
    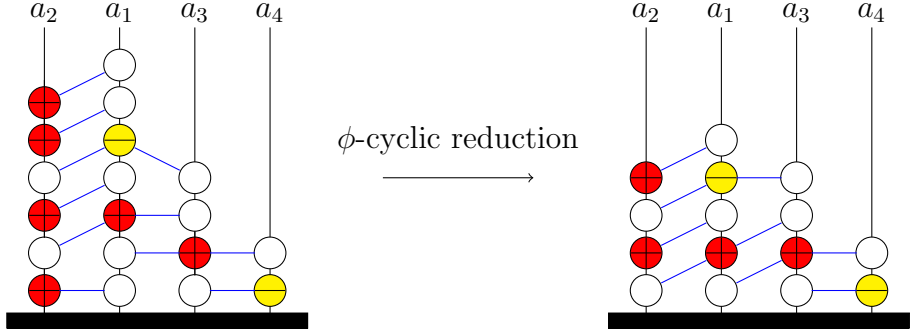
\begin{figure}[h]
        \centering
        \begin{tikzpicture}
            \filldraw[black] (0,0) rectangle (4, 0.2);
            \draw[black] (0.5, 0.2) -- (0.5, 4);
            \node at (0.5, 4.2) {$a_{2}$};
            \node at (1.5, 4.2) {$a_{1}$};
            \node at (2.5, 4.2) {$a_{3}$};
            \node at (3.5, 4.2) {$a_{4}$};
            \draw[black] (1.5,0.2) -- (1.5,4) ;
            \draw[black] (2.5,0.2) -- (2.5,4) ;
            \draw[black] (3.5,0.2) -- (3.5,4) ;
            \filldraw[black] (8,0) rectangle (12, 0.2);
            \node at (8.5, 4.2) {$a_{2}$};
            \node at (9.5, 4.2) {$a_{1}$};
            \node at (10.5, 4.2) {$a_{3}$};
            \node at (11.5, 4.2) {$a_{4}$};
            \draw[black] (8.5,0.2) -- (8.5,4) ;
            \draw[black] (9.5,0.2) -- (9.5,4) ;
            \draw[black] (10.5,0.2) -- (10.5,4) ;
            \draw[black] (11.5,0.2) -- (11.5,4) ;
            \draw[->] (5, 2) -- (7,2);
            \node at (6,2.5) {$\phi$-cyclic reduction};
            \filldraw[color = black, fill=red] (0.5, 0.5) circle (6pt);
            \draw[black] (0.3, 0.5) -- (0.7, 0.5);
            \draw[black] (0.5, 0.2) -- (0.5, 0.8);
            \filldraw[color = black, fill=white] (1.5, 0.5) circle (6pt);
            \filldraw[color = black, fill=white] (2.5, 0.5) circle (6pt);
            \filldraw[color = black, fill=yellow] (3.5, 0.5) circle (6pt);
            \draw[black] (3.3, 0.5) -- (3.7, 0.5);
            \filldraw[color = black, fill=white] (0.5, 1) circle (6pt);
            \filldraw[color = black, fill=white] (1.5, 1) circle (6pt);
            \filldraw[color = black, fill=red] (2.5, 1) circle (6pt);
            \draw[black] (2.3, 1) -- (2.7, 1);
            \draw[black] (2.5, 0.7) -- (2.5, 1.3);
            \filldraw[color = black, fill=white] (3.5, 1) circle (6pt);
            \filldraw[color = black, fill=red] (0.5, 1.5) circle (6pt);
            \draw[black] (0.3, 1.5) -- (0.7, 1.5);
            \draw[black] (0.5, 1.2) -- (0.5, 1.8);
            \filldraw[color = black, fill=red] (1.5, 1.5) circle (6pt);
            \draw[black] (1.3, 1.5) -- (1.7, 1.5);
            \draw[black] (1.5, 1.2) -- (1.5, 1.8) ;
            \filldraw[color = black, fill=white] (2.5, 1.5) circle (6pt);
            \filldraw[color = black, fill=white] (0.5, 2) circle (6pt);
            \filldraw[color = black, fill=white] (1.5, 2) circle (6pt);
            \filldraw[color = black, fill=white] (2.5, 2) circle (6pt);
            \filldraw[color = black, fill=red] (0.5, 2.5) circle (6pt);
            \draw[black] (0.3, 2.5) -- (0.7, 2.5);
            \draw[black] (0.5, 2.3) -- (0.5, 2.8);
            \filldraw[color = black, fill=yellow] (1.5, 2.5) circle (6pt);
            \draw[black] (1.3, 2.5) -- (1.7, 2.5);
            \filldraw[color = black, fill=red] (0.5, 3) circle (6pt);
            \draw[black] (0.3, 3) -- (0.7, 3);
            \draw[black] (0.5, 2.7) -- (0.5, 3.3);
            \filldraw[color = black, fill=white] (1.5, 3) circle (6pt);
            \filldraw[color = black, fill=white] (1.5, 3.5) circle (6pt);
            \draw[blue] (0.7, 3.1) -- (1.3, 3.4);
            \draw[blue] (0.7, 2.6) -- (1.3, 2.9);
            \draw[blue] (0.7, 2.1) -- (1.3, 2.4);
            \draw[blue] (1.7, 2.4) -- (2.3, 2.1) ;
            \draw[blue] (0.7, 1.6) -- (1.3, 1.9);
            \draw[blue] (0.7, 1.1) -- (1.3, 1.4);
            \draw[blue] (0.7, 0.5)--(1.3, 0.5);
            \draw[blue] (1.7, 1.5) -- (2.3, 1.5);
            \draw[blue] (1.7, 1) -- (2.3, 1);
            \draw[blue] (2.7, 1) -- (3.3, 1);
            \draw[blue] (2.7, 0.5) -- (3.3, 0.5);
            \filldraw[color = black, fill = white] (8.5, 0.5) circle (6pt);
            \filldraw[color = black, fill = white] (9.5, 0.5) circle (6pt);
            \filldraw[color = black, fill = white] (10.5, 0.5) circle (6pt);
            \filldraw[color = black, fill = yellow] (11.5, 0.5) circle (6pt);
            \draw[blue] (10.7, 0.5) -- (11.3, 0.5);
            \draw[black] (11.3, 0.5)--(11.7, 0.5);
            \filldraw[color = black, fill = red] (8.5, 1) circle (6pt);
            \filldraw[color = black, fill = red] (9.5, 1) circle (6pt);
            \filldraw[color = black, fill = red] (10.5, 1) circle (6pt);
            \filldraw[color = black, fill = white] (11.5, 1) circle (6pt);
            \draw[black] (8.3, 1) -- (8.7, 1);
            \draw[black] (8.5, 0.8) -- (8.5, 1.2) ; 
            \draw[black] (9.3, 1) -- (9.7, 1);
            \draw[black] (9.5, 0.8) -- (9.5, 1.2) ; 
            \draw[black] (10.3, 1) -- (10.7, 1);
            \draw[black] (10.5, 0.8) -- (10.5, 1.2) ; 
            \draw[blue] (8.7, 0.6) -- (9.3, 0.9);
            \draw[blue] (9.7, 0.6) -- (10.3, 0.9);
            \draw[blue] (10.7, 1) -- (11.3, 1);
            \filldraw[color = black, fill = white] (8.5, 1.5) circle (6pt);
            \filldraw[color = black, fill = white] (9.5, 1.5) circle (6pt);
            \filldraw[color = black, fill = white] (10.5, 1.5) circle (6pt);
            \draw[blue] (8.7, 1.1) -- (9.3, 1.4);
            \draw[blue] (9.7, 1.1) -- (10.3, 1.4);
            \filldraw[color=black, fill = red] (8.5, 2) circle (6pt);
            \filldraw[color=black, fill = yellow] (9.5, 2) circle (6pt);
            \filldraw[color=black, fill = white] (10.5, 2) circle (6pt);
            \draw[black] (8.3, 2) -- (8.7, 2);
            \draw[black] (8.5, 1.8) -- (8.5, 2.2);
            \draw[black] (9.3, 2) -- (9.7, 2);
            \draw[blue] (8.7, 1.6) -- (9.3, 1.9);
            \draw[blue] (9.7, 2) -- (10.3, 2);
            \filldraw[color=black, fill = white] (9.5, 2.5) circle (6pt);
            \draw[blue] (8.7, 2.1) -- (9.3, 2.4);
        \end{tikzpicture}
        \caption{$\phi$-CR for inversions.}
        \label{fig:twisted CR inversions}
    \end{figure}
\end{exmp}
We now define $\phi$-cyclic permutations in pilings, to coincide with \cref{defn:cyc perm}. For twisted conjugacy, when moving a tile from the bottom to the top of our piling (or vice versa), we need to also apply $\phi$ to the corresponding $+$ or $-$ bead. 
\begin{defn}
    Let $p \in P$. We define a $\phi$-cyclic permutation of $p$ as follows:
    \begin{enumerate}
    \item If $\phi(s_{i}) = s_{i}^{-1}$ and the $\sigma_{i}$-stack starts with a $+$ bead, then remove the tile corresponding to this $+$ bead, and add a $-$ bead, with corresponding 0-beads, to the top of the $\sigma_{i}$-stack. Similarly if the $\sigma_{i}$-stack starts with a $-$ bead, then remove the tile corresponding to this $-$ bead, and add a $+$ bead, with corresponding 0-beads, to the top of the $\sigma_{i}$-stack. This definition is analogous for moving tiles from the top to the bottom of the piling.
    \item Otherwise, if the $\sigma_{i}$-stack starts or ends with a $+$ or $-$ bead, then apply a cyclic permutation as given in \cref{defn:piling cyclic perm}. 
\end{enumerate}

\end{defn}
\begin{exmp}
    We recall \cref{exmp: inversions}, and consider a $\phi$-cyclic permutation of our $\phi$-CR piling. \cref{fig:twisted perm inversion} gives an example of a $\phi$-cyclic permutation. As words, we are applying the operation
\[ u =_{\RAAG} a^{-1}_{4}a_{3}a_{1}a_{2}a^{-1}_{1}a_{2} \xrightarrow{\phi-\text{CP}}  a_{3}a_{1}a_{2}a^{-1}_{1}a_{2}\phi\left(a^{-1}_{4}\right) =  a_{3}a_{1}a_{2}a^{-1}_{1}a_{2}a_{4}.
\]
    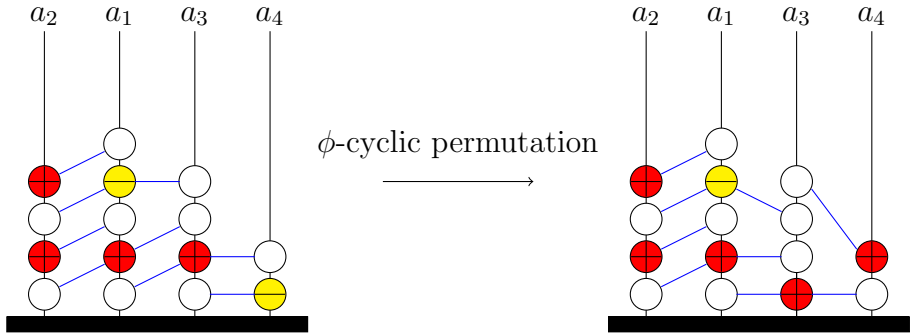
\begin{figure}[h]
        \centering
        \begin{tikzpicture}
            \filldraw[black] (0,0) rectangle (4, 0.2);
            \draw[black] (0.5, 0.2) -- (0.5, 4);
            \node at (0.5, 4.2) {$a_{2}$};
            \node at (1.5, 4.2) {$a_{1}$};
            \node at (2.5, 4.2) {$a_{3}$};
            \node at (3.5, 4.2) {$a_{4}$};
            \draw[black] (1.5,0.2) -- (1.5,4) ;
            \draw[black] (2.5,0.2) -- (2.5,4) ;
            \draw[black] (3.5,0.2) -- (3.5,4) ;
            \filldraw[black] (8,0) rectangle (12, 0.2);
            \node at (8.5, 4.2) {$a_{2}$};
            \node at (9.5, 4.2) {$a_{1}$};
            \node at (10.5, 4.2) {$a_{3}$};
            \node at (11.5, 4.2) {$a_{4}$};
            \draw[black] (8.5,0.2) -- (8.5,4) ;
            \draw[black] (9.5,0.2) -- (9.5,4) ;
            \draw[black] (10.5,0.2) -- (10.5,4) ;
            \draw[black] (11.5,0.2) -- (11.5,4) ;
            \draw[->] (5, 2) -- (7,2);
            \node at (6,2.5) {$\phi$-cyclic permutation};
            \filldraw[color = black, fill = white] (0.5, 0.5) circle (6pt);
            \filldraw[color = black, fill = white] (1.5, 0.5) circle (6pt);
            \filldraw[color = black, fill = white] (2.5, 0.5) circle (6pt);
            \filldraw[color = black, fill = yellow] (3.5, 0.5) circle (6pt);
            \draw[blue] (2.7, 0.5) -- (3.3, 0.5);
            \draw[black] (3.3, 0.5)--(3.7, 0.5);
            \filldraw[color = black, fill = red] (0.5, 1) circle (6pt);
            \filldraw[color = black, fill = red] (1.5, 1) circle (6pt);
            \filldraw[color = black, fill = red] (2.5, 1) circle (6pt);
            \filldraw[color = black, fill = white] (3.5, 1) circle (6pt);
            \draw[black] (0.3, 1) -- (0.7, 1);
            \draw[black] (0.5, 0.8) -- (0.5, 1.2) ; 
            \draw[black] (1.3, 1) -- (1.7, 1);
            \draw[black] (1.5, 0.8) -- (1.5, 1.2) ; 
            \draw[black] (2.3, 1) -- (2.7, 1);
            \draw[black] (2.5, 0.8) -- (2.5, 1.2) ; 
            \draw[blue] (0.7, 0.6) -- (1.3, 0.9);
            \draw[blue] (1.7, 0.6) -- (2.3, 0.9);
            \draw[blue] (2.7, 1) -- (3.3, 1);
            \filldraw[color = black, fill = white] (0.5, 1.5) circle (6pt);
            \filldraw[color = black, fill = white] (1.5, 1.5) circle (6pt);
            \filldraw[color = black, fill = white] (2.5, 1.5) circle (6pt);
            \draw[blue] (0.7, 1.1) -- (1.3, 1.4);
            \draw[blue] (1.7, 1.1) -- (2.3, 1.4);
            \filldraw[color=black, fill = red] (0.5, 2) circle (6pt);
            \filldraw[color=black, fill = yellow] (1.5, 2) circle (6pt);
            \filldraw[color=black, fill = white] (2.5, 2) circle (6pt);
            \draw[black] (0.3, 2) -- (0.7, 2);
            \draw[black] (0.5, 1.8) -- (0.5, 2.2);
            \draw[black] (1.3, 2) -- (1.7, 2);
            \draw[blue] (0.7, 1.6) -- (1.3, 1.9);
            \draw[blue] (1.7, 2) -- (2.3, 2);
            \filldraw[color=black, fill = white] (1.5, 2.5) circle (6pt);
            \draw[blue] (0.7, 2.1) -- (1.3, 2.4);
            \filldraw[color = black, fill = white] (8.5, 0.5) circle (6pt);
            \filldraw[color = black, fill = white] (9.5, 0.5) circle (6pt);
            \filldraw[color = black, fill = red] (10.5, 0.5) circle (6pt);
            \filldraw[color = black, fill = white] (11.5, 0.5) circle (6pt);
            \draw[blue] (10.7, 0.5) -- (11.3, 0.5);
             \draw[blue] (9.7, 0.5) -- (10.3, 0.5);
            \draw[black] (10.3, 0.5)--(10.7, 0.5);
            \draw[black] (10.5, 0.2) -- (10.5, 0.8);
            \filldraw[color = black, fill = red] (8.5, 1) circle (6pt);
            \filldraw[color = black, fill = red] (9.5, 1) circle (6pt);
            \filldraw[color = black, fill = white] (10.5, 1) circle (6pt);
            \filldraw[color = black, fill = red] (11.5, 1) circle (6pt);
            \draw[black] (8.3, 1) -- (8.7, 1);
            \draw[black] (8.5, 0.8) -- (8.5, 1.2) ; 
            \draw[black] (9.3, 1) -- (9.7, 1);
            \draw[black] (9.5, 0.8) -- (9.5, 1.2) ; 
            \draw[black] (11.3, 1) -- (11.7, 1);
            \draw[black] (11.5, 0.8) -- (11.5, 1.2);
            \draw[blue] (8.7, 0.6) -- (9.3, 0.9);
            \draw[blue] (9.7, 1) -- (10.3, 1);
            \filldraw[color = black, fill = white] (8.5, 1.5) circle (6pt);
            \filldraw[color = black, fill = white] (9.5, 1.5) circle (6pt);
            \filldraw[color = black, fill = white] (10.5, 1.5) circle (6pt);
            \draw[blue] (8.7, 1.1) -- (9.3, 1.4);
            \filldraw[color=black, fill = red] (8.5, 2) circle (6pt);
            \filldraw[color=black, fill = yellow] (9.5, 2) circle (6pt);
            \filldraw[color=black, fill = white] (10.5, 2) circle (6pt);
            \draw[black] (8.3, 2) -- (8.7, 2);
            \draw[black] (8.5, 1.8) -- (8.5, 2.2);
            \draw[black] (9.3, 2) -- (9.7, 2);
            \draw[blue] (8.7, 1.6) -- (9.3, 1.9);
            \draw[blue] (9.7, 1.9) -- (10.3, 1.6);
            \filldraw[color=black, fill = white] (9.5, 2.5) circle (6pt);
            \draw[blue] (8.7, 2.1) -- (9.3, 2.4);
            \draw[blue] (10.7, 1.9) -- (11.3, 1.1);
        \end{tikzpicture}
        \caption{$\phi$-cyclic permutation for inversion}
        \label{fig:twisted perm inversion}
    \end{figure}
\end{exmp}
For standard conjugacy, a cyclic normal form is obtained from pyramidal pilings. We now define a similar cyclic normal form, which is preserved under twisted conjugacy. 
\begin{defn}\label{defn:phi cyclic normal form}
    Let $w \in X^{\ast}$ be reduced and $\phi$-cyclically reduced. We say $w$ is a \emph{$\phi$-cyclic normal form} if it is normal and all $\phi$-cyclic permutations are also normal.
\end{defn}
When we apply a $\phi$-cyclic permutation to a pyramidal piling, our $\phi$-cyclic normal form is preserved. This is due to the fact that $s_{i} \leq^{-1}_{\mathrm{SL}} s^{\pm 1}_{j}$ if and only if $s^{-1}_{i} \leq^{-1}_{\mathrm{SL}} s^{\pm 1}_{j}$, for all $1 \leq i, j \leq r$. Therefore if $w = s^{\pm 1}_{j}us^{\pm 1}_{i}$ is a normal form obtained from a pyramidal piling, where $u \in X^{\ast}$, and we consider a $\phi$-cyclic permutation $w' = \phi\left(s^{\pm 1}_{i}\right)s^{\pm 1}_{j}u = s^{\pm 1}_{i}s^{\pm 1}_{j}u$, then $w'$ is also a normal form. The same is true when we apply a $\phi$-cyclic permutation of a bottom tile of a pyramidal piling. This allows us to adapt \cref{prop: 2.20 RAAGs} and \cref{prop: 2.21 RAAGs} with respect to twisted conjugacy by inversions.
\begin{prop}
    Let $p \in P$ be a non-split $\phi$-cyclically reduced pyramidal piling. Let $w \in X^{\ast}$ be the unique normal reduced word which represent the group element defined by $p$. Then $w$ is a $\phi$-cyclic normal form.
\end{prop}
\begin{prop}
    Two $\phi$-cyclic normal forms represent conjugate elements if and only if they are equal up to a $\phi$-cyclic permutation. 
\end{prop}
We now have the necessary definitions and results to prove the following.

\begin{thm}\label{thm:TCP inversions RAAGs linear}
    Let $\phi \in \mathrm{Aut}(A_{\Gamma})$ be a composition of inversions. Then the $\mathrm{TCP}_{\phi}(A_{\Gamma})$ is solvable in linear time.
\end{thm}

\begin{proof}
    We summarise our algorithm here, which is an adapted version of the conjugacy problem in RAAGs. 
    
    \textbf{\underline{Algorithm: Twisted conjugacy problem in RAAGs (inversions)}}
    
    \textbf{Input}:
\begin{enumerate}
    \item RAAG $A_{\Gamma}$.
    \item Words $v,w \in X^{\ast}$ representing group elements in $A_{\Gamma}$.
    \item $\phi \in \mathrm{Aut}(\RAAG)$ given as a composition of inversions. 
\end{enumerate}
\textbf{Step 1: $\phi$-cyclic reduction}
\begin{adjustwidth}{1.5cm}{}
    Produce the piling representation $\pi^{\ast}(v)$ of the word $v$, and apply $\phi$-cyclic reduction to $\pi^{\ast}(v)$ to produce a $\phi$-cyclically reduced piling $p$. Repeat this step for the word $w$ to get a $\phi$-cyclically reduced piling $q$.
\end{adjustwidth}
\textbf{Step 2: Factorisation}
\begin{adjustwidth}{1.5cm}{}
    Factorise each of the pilings $p$ and $q$ into non-split factors. If the collection of subgraphs do not coincide, \textbf{Output} = \texttt{\color{blue}False}.
\end{adjustwidth}
\textbf{Step 3: Compare non-split factors}
\begin{adjustwidth}{1.5cm}{}
    If $p = p^{(1)}\dots p^{(k)}$ and $q = q^{(1)}\dots q^{(k)}$ are the factorisations found in Step 2, then for each $i = 1, \dots, k$, do the following:
    \begin{adjustwidth}{1.5cm}{}
        \begin{enumerate}[label=(\roman*)]
            \item Transform the non-split $\phi$-cyclically reduced pilings $p^{(i)}$ and $q^{(i)}$ into pyramidal pilings $\Tilde{p}^{(i)}$ and $\Tilde{q}^{(i)}$, by applying $\phi$-cyclic permutations.
            \item Produce the words representing these pilings in $\phi$-cyclic normal form $\sigma^{\ast}\left(\Tilde{p}^{(i)}\right)$ and $\sigma^{\ast}\left(\Tilde{q}^{(i)}\right)$.
            \item Decide whether $\sigma^{\ast}\left(\Tilde{p}^{(i)}\right)$ and $\sigma^{\ast}\left(\Tilde{q}^{(i)}\right)$ are equal up to a $\phi$-cyclic permutation. If not, \textbf{Output} = \texttt{\color{blue}False}.
        \end{enumerate}
    \end{adjustwidth}
    \textbf{Output} = \texttt{\color{blue}True}.
\end{adjustwidth}
\vspace{-20pt}
\end{proof}

\subsubsection{Length-preserving automorphisms}
We now consider all length-preserving automorphisms. Again we can use the pilings construction, and define $\phi$-CR and $\phi$-cyclic permutations with respect to graph automorphisms.
\begin{defn}
    Let $\phi \in \mathrm{Aut}(\RAAG)$ be a graph automorphism, and let $p \in P$ be a non-empty piling. We define a $\phi$-cyclic reduction, with respect to graph automorphisms, as follows. Consider a $\sigma_{i}$-stack such that $\phi(s_{i}) = s_{j}$, for some $1 \leq i,j \leq r$. If the $\sigma_{i}$-stack ends with a $+$ bead (resp. $-$ bead), and the $\sigma_{j}$-stack starts with a $-$ bead (resp. $+$ bead), then remove the top (resp. bottom) tile of $\sigma_{i}$ and bottom (resp. top) tile of $\sigma_{j}$ (with corresponding 0-beads from these tiles).

    We say $p$ is $\phi$-CR if no $\phi$-cyclic reduction is possible. Note if $j=i$, i.e. $s_{i}$ is fixed under $\phi$, this definition is equivalent to cyclic reduction as given in \cref{defn:CR piling}.
\end{defn}
\begin{defn}
    Let $\phi \in \mathrm{Aut}(\RAAG)$ be a graph automorphism, and let $p \in P$ be a non-empty piling. We define a $\phi$-cyclic permutation for graph automorphisms as follows. For each $\sigma_{i}$-stack such that $\phi(s_{i}) = s_{j}$ ($1 \leq i,j \leq r$), we can apply either of the following:
    \begin{enumerate}
        \item If the $\sigma_{i}$-stack ends with a $+$ bead (resp. $-$ bead), remove this tile, and add a $+$ tile (resp. $-$ tile) to the bottom of the $\sigma_{j}$-stack. This is equivalent to a $\phi$-cyclic permutation of the letter $s_{i}$ from the end to the start of the word $w \in X^{\ast}$ representing the piling. 
        \item If the $\sigma_{j}$-stack starts with a $+$ bead (resp $-$ bead), remove this tile, and add a $+$ tile (resp $-$ tile) to the top of the $\sigma_{i}$-stack. This is equivalent to a $\phi$-cyclic permutation of the letter $s_{j}$ from the start to the end of the word $w \in X^{\ast}$ representing the piling. 
    \end{enumerate}
    Note if $j=i$, this definition is equivalent to cyclic permutations as given in \cref{defn:piling cyclic perm}.
\end{defn}
\begin{exmp}\label{exmp:pilings graph auto operations}
    We return to \cref{exmp: inversions}, and consider the order two graph automorphism $\phi \colon a_{1} \leftrightarrow a_{3}, a_{2} \leftrightarrow a_{4}$. \cref{fig:phi-CR graph auto} gives an example of a $\phi$-CR, where as words we have
    \[ u =_{\RAAG} a^{-1}_{4}a_{2}a_{3}a_{1}a_{2}a^{-1}_{1}a_{2}a_{2} = \phi(a_{2})^{-1}a_{2}a_{3}a_{1}a_{2}a^{-1}_{1}a_{2}a_{2} \xrightarrow{\phi-\text{CR}} a_{2}a_{3}a_{1}a_{2}a^{-1}_{1}a_{2}.
    \]
    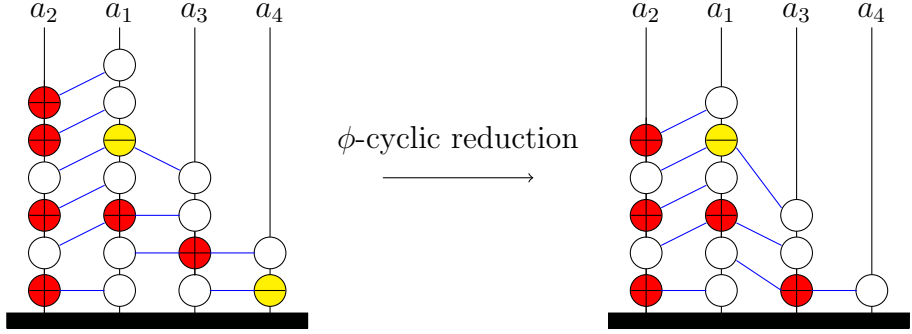
\begin{figure}[h]
        \centering
        \begin{tikzpicture}
            \filldraw[black] (0,0) rectangle (4, 0.2);
            \draw[black] (0.5, 0.2) -- (0.5, 4);
            \node at (0.5, 4.2) {$a_{2}$};
            \node at (1.5, 4.2) {$a_{1}$};
            \node at (2.5, 4.2) {$a_{3}$};
            \node at (3.5, 4.2) {$a_{4}$};
            \draw[black] (1.5,0.2) -- (1.5,4) ;
            \draw[black] (2.5,0.2) -- (2.5,4) ;
            \draw[black] (3.5,0.2) -- (3.5,4) ;
            \filldraw[black] (8,0) rectangle (12, 0.2);
            \node at (8.5, 4.2) {$a_{2}$};
            \node at (9.5, 4.2) {$a_{1}$};
            \node at (10.5, 4.2) {$a_{3}$};
            \node at (11.5, 4.2) {$a_{4}$};
            \draw[black] (8.5,0.2) -- (8.5,4) ;
            \draw[black] (9.5,0.2) -- (9.5,4) ;
            \draw[black] (10.5,0.2) -- (10.5,4) ;
            \draw[black] (11.5,0.2) -- (11.5,4) ;
            \draw[->] (5, 2) -- (7,2);
            \node at (6,2.5) {$\phi$-cyclic reduction};
            \filldraw[color = black, fill=red] (0.5, 0.5) circle (6pt);
            \draw[black] (0.3, 0.5) -- (0.7, 0.5);
            \draw[black] (0.5, 0.2) -- (0.5, 0.8);
            \filldraw[color = black, fill=white] (1.5, 0.5) circle (6pt);
            \filldraw[color = black, fill=white] (2.5, 0.5) circle (6pt);
            \filldraw[color = black, fill=yellow] (3.5, 0.5) circle (6pt);
            \draw[black] (3.3, 0.5) -- (3.7, 0.5);
            \filldraw[color = black, fill=white] (0.5, 1) circle (6pt);
            \filldraw[color = black, fill=white] (1.5, 1) circle (6pt);
            \filldraw[color = black, fill=red] (2.5, 1) circle (6pt);
            \draw[black] (2.3, 1) -- (2.7, 1);
            \draw[black] (2.5, 0.7) -- (2.5, 1.3);
            \filldraw[color = black, fill=white] (3.5, 1) circle (6pt);
            \filldraw[color = black, fill=red] (0.5, 1.5) circle (6pt);
            \draw[black] (0.3, 1.5) -- (0.7, 1.5);
            \draw[black] (0.5, 1.2) -- (0.5, 1.8);
            \filldraw[color = black, fill=red] (1.5, 1.5) circle (6pt);
            \draw[black] (1.3, 1.5) -- (1.7, 1.5);
            \draw[black] (1.5, 1.2) -- (1.5, 1.8) ;
            \filldraw[color = black, fill=white] (2.5, 1.5) circle (6pt);
            \filldraw[color = black, fill=white] (0.5, 2) circle (6pt);
            \filldraw[color = black, fill=white] (1.5, 2) circle (6pt);
            \filldraw[color = black, fill=white] (2.5, 2) circle (6pt);
            \filldraw[color = black, fill=red] (0.5, 2.5) circle (6pt);
            \draw[black] (0.3, 2.5) -- (0.7, 2.5);
            \draw[black] (0.5, 2.3) -- (0.5, 2.8);
            \filldraw[color = black, fill=yellow] (1.5, 2.5) circle (6pt);
            \draw[black] (1.3, 2.5) -- (1.7, 2.5);
            \filldraw[color = black, fill=red] (0.5, 3) circle (6pt);
            \draw[black] (0.3, 3) -- (0.7, 3);
            \draw[black] (0.5, 2.7) -- (0.5, 3.3);
            \filldraw[color = black, fill=white] (1.5, 3) circle (6pt);
            \filldraw[color = black, fill=white] (1.5, 3.5) circle (6pt);
            \draw[blue] (0.7, 3.1) -- (1.3, 3.4);
            \draw[blue] (0.7, 2.6) -- (1.3, 2.9);
            \draw[blue] (0.7, 2.1) -- (1.3, 2.4);
            \draw[blue] (1.7, 2.4) -- (2.3, 2.1) ;
            \draw[blue] (0.7, 1.6) -- (1.3, 1.9);
            \draw[blue] (0.7, 1.1) -- (1.3, 1.4);
            \draw[blue] (0.7, 0.5)--(1.3, 0.5);
            \draw[blue] (1.7, 1.5) -- (2.3, 1.5);
            \draw[blue] (1.7, 1) -- (2.3, 1);
            \draw[blue] (2.7, 1) -- (3.3, 1);
            \draw[blue] (2.7, 0.5) -- (3.3, 0.5);
            \filldraw[color = black, fill = red] (8.5, 0.5) circle (6pt);
            \filldraw[color = black, fill = white] (9.5, 0.5) circle (6pt);
            \filldraw[color = black, fill = red] (10.5, 0.5) circle (6pt);
            \filldraw[color = black, fill = white] (11.5, 0.5) circle (6pt);
            \draw[blue] (10.7, 0.5) -- (11.3, 0.5);
            \draw[blue] (9.7, 0.9) -- (10.3, 0.5);
            \draw[black] (10.3, 0.5)--(10.7, 0.5);
            \draw[black] (10.5, 0.2) -- (10.5, 0.8);
            \draw[black] (8.3, 0.5)--(8.7, 0.5);
            \draw[black] (8.5, 0.2) -- (8.5, 0.8);
            \filldraw[color = black, fill = white] (8.5, 1) circle (6pt);
            \filldraw[color = black, fill = white] (9.5, 1) circle (6pt);
            \filldraw[color = black, fill = white] (10.5, 1) circle (6pt);
            \draw[blue] (8.7, 0.5) -- (9.3, 0.5);
            \draw[blue] (8.7, 2.6) -- (9.3, 2.9);
            \filldraw[color = black, fill = red] (8.5, 1.5) circle (6pt);
            \filldraw[color = black, fill = red] (9.5, 1.5) circle (6pt);
            \filldraw[color = black, fill = white] (10.5, 1.5) circle (6pt);
            \draw[blue] (8.7, 1.1) -- (9.3, 1.4);
            \draw[blue] (9.7, 1.4) -- (10.3, 1.1);
            \draw[black] (8.3, 1.5) -- (8.7, 1.5);
            \draw[black] (8.5, 1.2) -- (8.5, 1.8) ;
            \draw[black] (9.3, 1.5) -- (9.7, 1.5);
            \draw[black] (9.5, 1.2) -- (9.5, 1.8) ;
            \filldraw[color=black, fill = white] (8.5, 2) circle (6pt);
            \filldraw[color=black, fill = white] (9.5, 2) circle (6pt);
            \draw[blue] (8.7, 1.6) -- (9.3, 1.9);
            \filldraw[color=black, fill = red] (8.5, 2.5) circle (6pt);
            \filldraw[color=black, fill = yellow] (9.5, 2.5) circle (6pt);
            \draw[black] (8.3, 2.5) -- (8.7, 2.5);
            \draw[black] (8.5, 2.2) -- (8.5, 2.8) ;
            \draw[black] (9.3, 2.5) -- (9.7, 2.5) ;
            \draw[blue] (8.7, 2.1) -- (9.3, 2.4);
            \draw[blue] (9.7, 2.4) -- (10.3, 1.6);
            \filldraw[color=black, fill = white] (9.5, 3) circle (6pt);
        \end{tikzpicture}
        \caption{$\phi$-CR for graph automorphism.}
        \label{fig:phi-CR graph auto}
    \end{figure}\\
    \cref{fig:phi-cyclic permutation graph auto} gives an example of a $\phi$-cyclic permutation, where as words we are applying the operation
    \[ u =_{\RAAG} a_{3}a_{2}a_{1}a_{2}a^{-1}_{1}a_{2} \xrightarrow{\phi-\text{CP}} a_{2}a_{1}a_{2}a^{-1}_{1}a_{2}\phi(a_{3}) = a_{2}a_{1}a_{2}a^{-1}_{1}a_{2}a_{1}.
    \]
    
    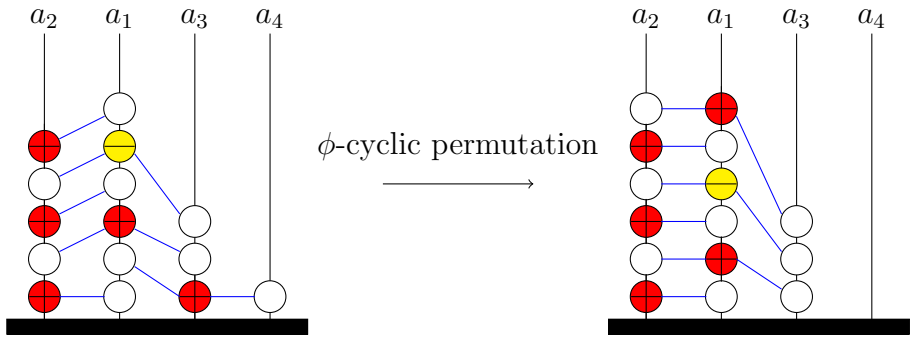
\begin{figure}[h]
        \centering
        \begin{tikzpicture}
            \filldraw[black] (0,0) rectangle (4, 0.2);
            \draw[black] (0.5, 0.2) -- (0.5, 4);
            \node at (0.5, 4.2) {$a_{2}$};
            \node at (1.5, 4.2) {$a_{1}$};
            \node at (2.5, 4.2) {$a_{3}$};
            \node at (3.5, 4.2) {$a_{4}$};
            \draw[black] (1.5,0.2) -- (1.5,4) ;
            \draw[black] (2.5,0.2) -- (2.5,4) ;
            \draw[black] (3.5,0.2) -- (3.5,4) ;
            \filldraw[black] (8,0) rectangle (12, 0.2);
            \node at (8.5, 4.2) {$a_{2}$};
            \node at (9.5, 4.2) {$a_{1}$};
            \node at (10.5, 4.2) {$a_{3}$};
            \node at (11.5, 4.2) {$a_{4}$};
            \draw[black] (8.5,0.2) -- (8.5,4) ;
            \draw[black] (9.5,0.2) -- (9.5,4) ;
            \draw[black] (10.5,0.2) -- (10.5,4) ;
            \draw[black] (11.5,0.2) -- (11.5,4) ;
            \draw[->] (5, 2) -- (7,2);
            \node at (6,2.5) {$\phi$-cyclic permutation};
            \filldraw[color = black, fill = red] (0.5, 0.5) circle (6pt);
            \filldraw[color = black, fill = white] (1.5, 0.5) circle (6pt);
            \filldraw[color = black, fill = red] (2.5, 0.5) circle (6pt);
            \filldraw[color = black, fill = white] (3.5, 0.5) circle (6pt);
            \draw[blue] (2.7, 0.5) -- (3.3, 0.5);
            \draw[blue] (1.7, 0.9) -- (2.3, 0.5);
            \draw[black] (2.3, 0.5)--(2.7, 0.5);
            \draw[black] (2.5, 0.2) -- (2.5, 0.8);
            \draw[black] (0.3, 0.5)--(0.7, 0.5);
            \draw[black] (0.5, 0.2) -- (0.5, 0.8);
            \filldraw[color = black, fill = white] (0.5, 1) circle (6pt);
            \filldraw[color = black, fill = white] (1.5, 1) circle (6pt);
            \filldraw[color = black, fill = white] (2.5, 1) circle (6pt);
            \draw[blue] (0.7, 0.5) -- (1.3, 0.5);
            \draw[blue] (0.7, 2.6) -- (1.3, 2.9);
            \filldraw[color = black, fill = red] (0.5, 1.5) circle (6pt);
            \filldraw[color = black, fill = red] (1.5, 1.5) circle (6pt);
            \filldraw[color = black, fill = white] (2.5, 1.5) circle (6pt);
            \draw[blue] (0.7, 1.1) -- (1.3, 1.4);
            \draw[blue] (1.7, 1.4) -- (2.3, 1.1);
            \draw[black] (0.3, 1.5) -- (0.7, 1.5);
            \draw[black] (0.5, 1.2) -- (0.5, 1.8) ;
            \draw[black] (1.3, 1.5) -- (1.7, 1.5);
            \draw[black] (1.5, 1.2) -- (1.5, 1.8) ;
            \filldraw[color=black, fill = white] (0.5, 2) circle (6pt);
            \filldraw[color=black, fill = white] (1.5, 2) circle (6pt);
            \draw[blue] (0.7, 1.6) -- (1.3, 1.9);
            \filldraw[color=black, fill = red] (0.5, 2.5) circle (6pt);
            \filldraw[color=black, fill = yellow] (1.5, 2.5) circle (6pt);
            \draw[black] (0.3, 2.5) -- (0.7, 2.5);
            \draw[black] (0.5, 2.2) -- (0.5, 2.8) ;
            \draw[black] (1.3, 2.5) -- (1.7, 2.5) ;
            \draw[blue] (0.7, 2.1) -- (1.3, 2.4);
            \draw[blue] (1.7, 2.4) -- (2.3, 1.6);
            \filldraw[color=black, fill = white] (1.5, 3) circle (6pt);
            \filldraw[color=black, fill = red] (8.5, 0.5) circle (6pt);
            \filldraw[color=black, fill = white] (8.5, 1) circle (6pt);
            \filldraw[color=black, fill = red] (8.5, 1.5) circle (6pt);
            \filldraw[color=black, fill = white] (8.5, 2) circle (6pt);
            \filldraw[color=black, fill = red] (8.5, 2.5) circle (6pt);
            \filldraw[color=black, fill = white] (8.5, 3) circle (6pt);
            \filldraw[color=black, fill = white] (9.5, 0.5) circle (6pt);
            \filldraw[color=black, fill = red] (9.5, 1) circle (6pt);
            \filldraw[color=black, fill = white] (9.5, 1.5) circle (6pt);
            \filldraw[color=black, fill = yellow] (9.5, 2) circle (6pt);
            \filldraw[color=black, fill = white] (9.5, 2.5) circle (6pt);
            \filldraw[color=black, fill = red] (9.5, 3) circle (6pt);
            \filldraw[color=black, fill = white] (10.5, 0.5) circle (6pt);
            \filldraw[color=black, fill = white] (10.5, 1) circle (6pt);
            \filldraw[color=black, fill = white] (10.5, 1.5) circle (6pt);
            \draw[black] (8.3, 0.5) -- (8.7, 0.5);
            \draw[black] (8.5, 0.2) -- (8.5, 0.8) ;
            \draw[black] (8.3, 1.5) -- (8.7, 1.5);
            \draw[black] (8.5, 1.2) -- (8.5, 1.8) ;
            \draw[black] (8.3, 2.5) -- (8.7, 2.5);
            \draw[black] (8.5, 2.2) -- (8.5, 2.8) ;
            \draw[black] (9.3, 1) -- (9.7, 1);
            \draw[black] (9.5, 0.7) -- (9.5, 1.3) ;
            \draw[black] (9.3, 2) -- (9.7, 2);
            \draw[black] (9.3, 3) -- (9.7, 3);
            \draw[black] (9.5, 2.7) -- (9.5, 3.3) ;
            \draw[blue] (8.7, 0.5) -- (9.3, 0.5);
            \draw[blue] (8.7, 1) -- (9.3, 1);
            \draw[blue] (9.7, 1) -- (10.3, 0.6);
            \draw[blue] (8.7, 1.5) -- (9.3, 1.5);
            \draw[blue] (8.7, 2) -- (9.3, 2);
            \draw[blue] (9.7, 1.9) -- (10.3, 1.1);
            \draw[blue] (8.7, 2.5) -- (9.3, 2.5);
            \draw[blue] (8.7, 3) -- (9.3, 3);
            \draw[blue] (9.7, 2.9) -- (10.3, 1.6);
        \end{tikzpicture}
        \caption{$\phi$-cyclic permutation for graph automorphism.}
        \label{fig:phi-cyclic permutation graph auto}
    \end{figure}
\end{exmp}

When we include graph automorphisms, we can no longer use the linear-time algorithm given in \cref{thm:TCP inversions RAAGs linear}, for several reasons. Firstly, $\phi$-CR can no longer be applied in linear time, since \cref{form:inversions} does not hold for graph automorphisms. This can be seen in \cite[Example 3.15]{crowe_conjugacy_2023}. Secondly, it is unclear how to factorise our defining graph into non-split factors. This is because when we apply $\phi$-cyclic permutations, we can obtain non-isomorphic induced subgraphs.

\begin{exmp}
    Recall \cref{exmp:pilings graph auto operations}. \cref{fig:complement graph auto example} gives the complement $\Gamma^{c}$ of the defining graph for $\RAAG$.
    \begin{figure}[h]
        \centering
        \begin{tikzpicture}
    \filldraw[black] (0,0) circle (2pt);
    \filldraw[black] (2,0) circle (2pt);
    \filldraw[black] (2,2) circle (2pt);
    \filldraw[black] (0,2) circle (2pt);
    \draw (0,0) -- (0,2);
    \draw (0,2) -- (2,2);
    \draw (2,2)-- (2,0);
    \node at (-0.3, 0) {$a_{2}$};
    \node at (-0.3, 2) {$a_{1}$};
    \node at (2.3, 0) {$a_{4}$};
    \node at (2.3, 2) {$a_{3}$};
    \end{tikzpicture}
        \caption{Complement of defining graph}
        \label{fig:complement graph auto example}
    \end{figure}
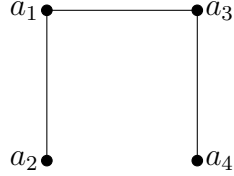\\
    \newpage
    Consider the words $u = a_{2}a_{1}$ and $v = a_{1}a_{4}$. These are related by a $\phi$-cyclic permutation, and \cref{fig:Non split and split issue} gives the corresponding piling representations. However, the induced subgraph $\Delta(u)$ on $\Gamma^{c}$ is connected, whereas the induced subgraph $\Delta(v)$ on $\Gamma^{c}$ is disconnected. Therefore $u$ is non-split, whereas $v$ is split. 
    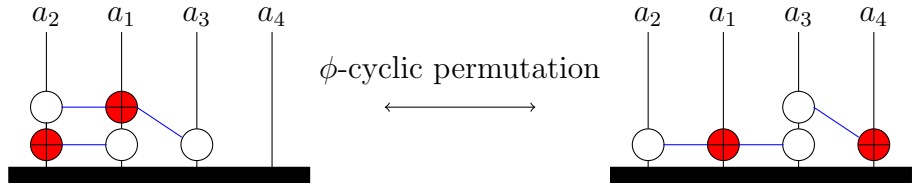
\begin{figure}[h]
        \centering
        \begin{tikzpicture}
            \filldraw[black] (0,0) rectangle (4, 0.2);
            \draw[black] (0.5, 0.2) -- (0.5, 2);
            \node at (0.5, 2.2) {$a_{2}$};
            \node at (1.5, 2.2) {$a_{1}$};
            \node at (2.5, 2.2) {$a_{3}$};
            \node at (3.5, 2.2) {$a_{4}$};
            \draw[black] (1.5,0.2) -- (1.5,2) ;
            \draw[black] (2.5,0.2) -- (2.5,2) ;
            \draw[black] (3.5,0.2) -- (3.5,2) ;
            \filldraw[black] (8,0) rectangle (12, 0.2);
            \node at (8.5, 2.2) {$a_{2}$};
            \node at (9.5, 2.2) {$a_{1}$};
            \node at (10.5, 2.2) {$a_{3}$};
            \node at (11.5, 2.2) {$a_{4}$};
            \draw[black] (8.5,0.2) -- (8.5,2) ;
            \draw[black] (9.5,0.2) -- (9.5,2) ;
            \draw[black] (10.5,0.2) -- (10.5,2) ;
            \draw[black] (11.5,0.2) -- (11.5,2) ;
            \draw[<->] (5, 1) -- (7,1);
            \node at (6,1.5) {$\phi$-cyclic permutation};
            \filldraw[color = black, fill = red] (0.5, 0.5) circle (6pt);
            \filldraw[color = black, fill = white] (1.5, 0.5) circle (6pt);
            \filldraw[color = black, fill = white] (2.5, 0.5) circle (6pt);
            \filldraw[color = black, fill = white] (0.5, 1) circle (6pt);
            \filldraw[color = black, fill = red] (1.5, 1) circle (6pt);
            \draw[blue] (0.7, 0.5) -- (1.3, 0.5);
            \draw[black] (0.3, 0.5) -- (0.7, 0.5);
            \draw[black] (0.5, 0.2) -- (0.5, 0.8);
            \draw[black] (1.3, 1) -- (1.7, 1);
            \draw[black] (1.5, 0.8) -- (1.5, 1.2);
            \draw[blue] (0.7, 1) -- (1.3, 1) ;
            \draw[blue] (1.7, 1) -- (2.3, 0.6);
            \filldraw[color=black, fill = white] (8.5, 0.5) circle (6pt);
            \filldraw[color=black, fill = red] (9.5, 0.5) circle (6pt);
            \filldraw[color=black, fill = white] (10.5, 0.5) circle (6pt);
            \filldraw[color=black, fill = white] (10.5, 1) circle (6pt);
            \filldraw[color=black, fill = red] (11.5, 0.5) circle (6pt);
            \draw[black] (9.3, 0.5) -- (9.7, 0.5);
            \draw[black] (9.5, 0.3) -- (9.5, 0.7);
            \draw[black] (11.3, 0.5) -- (11.7, 0.5);
            \draw[black] (11.5, 0.3) -- (11.5, 0.7);
            \draw[blue] (8.7, 0.5) -- (9.3, 0.5);
            \draw[blue] (9.7, 0.5) -- (10.3, 0.5);
            \draw[blue] (10.7, 1) -- (11.3, 0.6);
            \end{tikzpicture}
        \caption{Related split/non-split pilings}
        \label{fig:Non split and split issue}
    \end{figure}
\end{exmp}
Therefore it is unclear whether a pyramidal piling is well defined in this situation, since after applying a $\phi$-cyclic permutation to a pyramidal piling, we may obtain a split piling. Finally, unlike inversions, $\phi$-cyclic permutations no longer preserve the ordering $\leq^{-1}_{\mathrm{SL}}$, and so we cannot obtain a $\phi$-cyclic normal form. 

However, we can use \cref{thm:language sequence} to construct an algorithm which solves the twisted conjugacy problem using the pilings construction. The idea is to apply all possible $\phi$-cyclic permutations to pilings, whilst also applying reductions when necessary, to create a set $\mathcal{D}$ of piling representatives, which all lie within the same twisted conjugacy class. It then remains to check that for any two words $u,v \in X^{\ast}$, their corresponding sets of piling representatives are equal. Equivalently, we can consider the unique smallest normal form representative from each set, with respect to $\leq^{-1}_{\mathrm{SL}}$, and check these are equal. 

\begin{rmk}
    If $\phi \in \mathrm{Aut}(\RAAG)$ is length-preserving, then $\phi$ must be a composition of inversions and graph automorphisms by \cref{prop:combo len p}. When considering $\phi$-cyclic reduction and $\phi$-cyclic permutations, we combine the definitions of these operations for inversions and graph automorphisms.
\end{rmk}

\begin{proof}[Proof of \cref{thm:len p RAAGs pilings}]
    We provide details of our twisted conjugacy algorithm for RAAGs, where our automorphism is length-preserving. 

    \textbf{\underline{Algorithm: Twisted conjugacy problem in RAAGs (length-preserving)}}
    
    \textbf{Input}:
\begin{enumerate}
    \item RAAG $A_{\Gamma}$.
    \item Words $v,w \in X^{\ast}$ representing group elements in $A_{\Gamma}$.
    \item $\phi \in \mathrm{Aut}(\RAAG)$ which is length-preserving. 
\end{enumerate}
\textbf{Step 1: Piling representations}
\begin{adjustwidth}{1.5cm}{}
    Produce the piling representation $\pi^{\ast}(v)$ of the word $v$, and similarly $\pi^{\ast}(w)$ of the word $w$.
\end{adjustwidth}
\textbf{Step 2: Set of minimal representatives}
\begin{adjustwidth}{1.5cm}{}
    Compute all possible $\phi$-cyclic permutations of $\pi^{\ast}(v)$. If the length of any piling is less than that of $\pi^{\ast}(v)$, then restart the process. Let $\mathcal{D}(v)$ denote the set of all piling representatives found from this process. Similarly compute the set of piling representatives $\mathcal{D}(w)$ with respect to $\pi^{\ast}(w)$.
\end{adjustwidth}
\textbf{Step 3: Compare sets}
\begin{adjustwidth}{1.5cm}{}
    Let $\overline{v}, \overline{w} \in P$ denote the $\leq^{-1}_{\mathrm{SL}}$ minimal pilings which lie in $\mathcal{D}(v)$ and $\mathcal{D}(w)$ respectively. If $\overline{v} = \overline{w}$, then \textbf{Output} = \texttt{\color{blue}True}. Otherwise, \textbf{Output} = \texttt{\color{blue}False}.
\end{adjustwidth}
\vspace{-20pt}
\end{proof}
This algorithm can be adapted for RACGs, where we note that length-preserving automorphisms are precisely the graph automorphisms induced by the defining graph.

\begin{rmk}
    We elaborate on Step 2, in particular computing the set $\mathcal{D}(v)$. First let $\mathcal{D}(v)$ be the empty set, and let $T = \{\pi^{\ast}(v)\}$. For each $x \in T$, we remove $x$ from $T$, add $x$ to $\mathcal{D}(v)$ (unless $x \in \mathcal{D}(v)$), and determine the set $b_{x}$ of bottom tiles in $x$. Then for each $y \in b_{x}$, we apply a $\phi$-cyclic permutation to $x$ with respect to the bottom tile $y$. This produces a new piling $p$, which we add to our set $T$ (unless $p \in T$). We repeat this process until $T$ is empty, and restart the process if any new piling produced has length less than the length of $\pi^{\ast}(v)$. Note $\mathcal{D}(v)$ is necessarily finite since $\phi$ is length-preserving, so the length of a piling does not increase during this process.
\end{rmk}
Whilst we have an implementable solution to the twisted conjugacy problem in RAAGs for all length-preserving automorphisms, the complexity of our algorithm is unclear when our automorphism includes a graph automorphism.
\begin{question}
    What is the computational complexity of $\mathrm{TCP}_{\phi}(\RAAG)$, when $\phi \in \mathrm{Aut}(\RAAG)$ is length-preserving and includes a graph automorphism? 
\end{question}
One may hope to extend this pilings construction to finite order non-length preserving automorphisms, using \cref{thm:language sequence}. When applying $\phi$-cyclic permutations, the size of our piling could change if we consider non-length preserving automorphisms, so it remains unclear how to create an algorithm which terminates using pilings.

\section*{Acknowledgments}
This paper forms part of the author's PhD thesis, and so they would like to primarily thank their former supervisor Laura Ciobanu. The author would also like to thank Andrew Duncan and Alessandro Sisto for their helpful comments and suggestions.

\bibliography{references}  
\bibliographystyle{plain}
\uppercase{\footnotesize{Department of Mathematics, University of Manchester M13 9PL, UK and the Heilbronn Institute for Mathematical Research, Bristol, UK}}
\par 
\textit{Email address:} \texttt{gemma.crowe@manchester.ac.uk}

\end{document}